\documentclass[11pt]{amsart}
\usepackage[toc,page]{appendix}
\usepackage{latexsym,amsfonts,amssymb,amsthm,amsmath,mathrsfs, color,amscd,graphicx,fullpage,fancyhdr,ulem} 
\usepackage[colorlinks=true,urlcolor=blue,
citecolor=red,linkcolor=blue,linktocpage,pdfpagelabels,bookmarksnumbered,bookmarksopen]{hyperref}
\usepackage{comment}

\excludecomment{comment} 
\excludecomment{uncomment} 

\newcommand{\s}[1]{{\mathcal #1}}

\newcommand{\bb}[1]{{\mathbb #1}}

\newcounter{hyp}
\newcommand{\T}{\mathbb{T}}
\newcommand{\TT}{ \mathbb{T}}
\newcommand{\R}{\mathbb{R}}
\newcommand{\RR}{\mathbb{R}}

\def\dd{\hspace{1pt}{\rm d}} 
\newcommand{\ds}{\displaystyle}

\newcommand{\da}{\downarrow}

\def\a{\alpha}

\def\d{\delta}

\def\g{\gamma}

\def\e{\varepsilon}
\def\eps{\varepsilon}

\renewcommand\epsilon{\varepsilon}

\newcommand{\cA}{{\mathcal A}}
\newcommand{\A}{{\mathcal A}}
\newcommand{\cB}{{\mathcal B}}
\newcommand{\B}{{\mathcal B}}

\newcommand{\K}{{\mathcal K}}

\def\be{\begin{equation}}
\def\ee{\end{equation}}
\def\ba{\begin{array}}
\def\ea{\end{array}}

\newcommand{\corrd}[1]{\textcolor{blue}{#1}} 
\usepackage{ulem}

\newcommand{\sL}{{\mathscr L}}

\newcommand{\M}{{\mathscr M}}
\newcommand{\sP}{{\mathscr P}}

\newcommand{\mres}{\mathbin{\vrule height 1.6ex depth 0pt width
0.13ex\vrule height 0.13ex depth 0pt width 1.3ex}}

\newtheorem{theorem}{Theorem} 
\newtheorem{corollary}[theorem]{Corollary}

\newtheorem{lemma}[theorem]{Lemma}
\newtheorem{proposition}[theorem]{Proposition}

\newtheorem{definition}[theorem]{Definition}

\newtheorem{remark}[theorem]{Remark}

\numberwithin{equation}{section}
\numberwithin{theorem}{section}

\title{The planning problem in Mean Field Games as regularized mass transport}

\author[P.J. Graber]{P. Jameson Graber}
\address{Department of Mathematics, Baylor University, One Bear Place, Waco, TX 97328, USA}
\email{jameson\_graber@baylor.edu}

\author[A.R. M\'esz\'aros]{Alp\'ar R. M\'esz\'aros}  
\address{Department of Mathematics, UCLA, 520 Portola Plaza, Los Angeles, CA 90095, USA}
\email{alpar@math.ucla.edu} 

\author[F. J. Silva]{Francisco J. Silva} 
\address{Institut de recherche XLIM-DMI, UMR-CNRS 7252 Facult\'e des sciences et techniques 
Universit\'e de Limoges, 87060 Limoges, France}
\email{francisco.silva@unilim.fr}

\author[D. Tonon]{Daniela Tonon}
\address{CEREMADE, CNRS, [UMR 7534], Universit\'e Paris-Dauphine, PSL Research University, 75016 Paris, France}
\email{tonon@ceremade.dauphine.fr}


\dedicatory{Version: \today}

\begin{document}
\maketitle

\begin{abstract}
In this paper, using variational approaches,  we investigate the first order planning problem arising in the theory of mean field games. We show the existence and uniqueness of weak solutions of the problem in the case of a large class of Hamiltonians with arbitrary superlinear order of growth at infinity and local coupling functions. We require the initial and final measures to be merely summable. At the same time (relying on the techniques developed recently in \cite{GraMes}), under stronger monotonicity and convexity conditions on the data, we obtain Sobolev estimates on the solutions of the planning problem both for space and time derivatives.

\end{abstract}

\tableofcontents

\section{Introduction} \label{sec:intro}

The purpose of this article is to study the first order planning problem in mean field game theory, which can be formulated as a system of nonlinear partial differential equations:
\begin{equation}\label{eq:planning_general}
\left\{
\begin{array}{ll}
-\partial_t u  + H(x,\nabla u)=f(x,m), & {\rm{in}}\  (0,T)\times\T^d,\\
\partial_t m  -\nabla\cdot (D_\xi H(x,\nabla u)m)=0, & {\rm{in}}\  (0,T)\times\T^d,\\
m(0,\cdot)=m_0,\ \ m(T,\cdot)=m_T, & {\rm{in}}\ \T^d.\\
\end{array}
\right.
\end{equation}
The data consist of probability measures $m_0,m_T\in\sP(\T^d)$, a fixed time horizon $T>0$, a coupling function $f:\T^d\times[0,+\infty)\to\R$ and a Hamiltonian $H:\T^d\times\R^d\to\R$.
Our aim is to find conditions on the data for which weak solutions to \eqref{eq:planning_general} can be shown to exist and are unique.

The theory of Mean Field Games (briefly MFG in what follows) was thrust into the limelight by the works of J.-M. Lasry and P.-L. Lions on the one hand (see \cite{LasLio06i,LasLio06ii,LasLio07}) and M. Huang, R. Malham\'e and P. Caines on the other (see \cite{HuaMalCai}). Their main motivation was to study limits of Nash equilibria of (stochastic or deterministic) differential games when the number of players tends to infinity. Since then, it has become a very lively and active branch of the theory nonlinear partial differential equations.
In addition to studying Nash equilibria, Lions \cite{Lions_course_planning} proposed a corresponding {\it planning problem}, in which a central planner would like to steer a population to a predetermined final configuration while still allowing individuals to choose their own strategies.

Let us give a simple, brief interpretation of System \eqref{eq:planning_general} in terms of large numbers of interacting agents.
The solution of the Hamilton-Jacobi Equation \eqref{eq:planning_general}(i) is supposed to be the value function for an optimal control problem of the form
\begin{equation}\label{prob:control}
\inf_{\a}\left\{\int_t^T \big[L(x(s),\a(s))+f(x(s),m(s,x(s)))\big]\dd s +u(T,x(T))\right\}=:u(t,x)
\end{equation}
subject to 
$$
\left\{
\begin{array}{ll}
x'(s)=\alpha(s), & s\in(t,T]\\
x(t)=x\in\T^d.
\end{array}
\right.
$$

Here the Lagrangian $L:\T^d\times\R^d\to\R$ is the Legendre-Fenchel transform of $H$ w.r.t.~the second variable. 
Formally the optimal strategy is given in feedback form, hence for the agent it is optimal to play  $-D_\xi H(x(s),\nabla u(s,x(s)))$. Having this velocity field as a drift, the evolution of the agents' density is given by the solution of the second  equation in \eqref{eq:planning_general}.
Then the coupling of the equations in \eqref{eq:planning_general} implies that every player is acting optimally with respect to the competing choices, i.e.~the game is in equilibrium.
We underline the fact that when considering `standard' mean field games, typically, the final cost $u(T,\cdot) = u_T$ is treated as given, along with an initial population density $m(0,\cdot) = m_0$.
For the planning problem, however, we fix a target population density $m(T,\cdot) = m_T$, leaving $u(T,\cdot)$ as an adjustable variable by which a central planner may determine the final outcome. Thus in the above control problem in particular $u(T,\cdot)$ is part of the problem itself.

As in many studies of mean field games, we restrict our attention to the case where $f(x,m)$ is an increasing function in the $m$ variable, namely $\partial_m f(x,\cdot)>0$ for all $x\in\T^d$.
We interpret this to mean that agents have a preference for low-density regions, i.e.~they want to avoid congestion.

In spite of the large number of studies available on MFG, the literature on System \eqref{eq:planning_general} is sparse.
The cases when $f$ is increasing and $D_{\xi\xi}^2 H(x,\cdot)>0$ (and both are smooth) are well-understood in the literature for both first order and second order non-degenerate models. In his lectures P.-L. Lions showed how to transform \eqref{eq:planning_general} into a uniformly elliptic system (thanks to smoothness assumptions on $f$ and $H$) on space-time, and he showed the existence of classical solutions. One can summarize these results as follows.

\begin{theorem}{\cite{Lions_course_planning}}
Let $m_0,m_T$ be strictly positive probability densities of class $C^{1,\alpha}(\T^d)$ ($0<\alpha<1$), let moreover  $f$ and $H$ be smooth such that $\partial_m f(x,\cdot)>0$ for all $x\in\T^d$ and $D_{\xi\xi}^2 H(x,\cdot)>0$. Then, there exists a unique solution $(u,m)\in C^{2,\alpha}([0,T]\times\T^d)\times C^{1,\alpha}([0,T]\times\T^d)$ to \eqref{eq:planning_general} (here the uniqueness of $u$ has to be understood modulo constants). Moreover, $(u,m)\in C^{\infty}((0,T)\times\T^d)\times C^{\infty}((0,T)\times\T^d).$ 
\end{theorem}
Similar results can be achieved for the case of nondegenerate second order model as well, for purely quadratic Hamiltonians, or which are close at infinity to purely quadratic ones. The techniques used in this case are slightly different and they rely on the Hopf-Lax transformation, which is possible because of the quadratic Hamiltonian structure.

Existence of weak solutions to \eqref{eq:planning_general} in this latter case of nondegenerate second order models were obtained by A. Porretta in \cite{Por13,Por14} using energy methods. These results can be summarized as follows.

\begin{theorem}(\cite[Theorem 1.3]{Por14}, \cite[Theorem 2]{Por13})
Let us consider the nondegenerate diffusive model and let $f(x,m)$ be continuous, and nondecreasing w.r.t.~the $m$ variable. Let $H(x,\xi)$ be $C^1$ and convex in $\xi$ with quadratic growth. Let $m_0,m_T\in C^1(T^d)$ be strictly positive probability densities. Then there exists a  weak solution $(u,m)\in L^2([0,T]; H^1(\T^d))\times C^0([0,T]; L^1(\T^d))$ to \eqref{eq:planning_general}. Moreover, if $H$ is strictly convex in the $p$ variable, the solution is unique (modulo constants in the case of $u$).
\end{theorem} 

In the recent paper \cite{OrrPorSav} C. Orrieri, A. Porretta and G. Savar\'e study weak solutions of System \eqref{eq:planning_general} set on the whole space $\R^d$ with Hamiltonians of quadratic growth and couplings with general growth. Their analysis relies on the variational structure of the problem and on a suitable weak theory -- which they develop in the paper -- of distributional sub-solutions and their traces of Hamilton-Jacobi equations with summable right hand sides.

 In \cite{GoSe}, D.A. Gomes and T. Seneci  explore displacement convexity properties, introduced by the Benamou-Brenier formulation for optimal transport,   in order to obtain a priori estimates for solutions of  \eqref{eq:planning_general}. Finally, some numerical aspects of the mean field planning problem were investigated by Y. Achdou, F. Camilli and I. Capuzzo-Dolcetta in \cite{AchCamCap12}. 

Our goal in this paper is to prove the existence and uniqueness of solutions of \eqref{eq:planning_general} for general Hamiltonians $H$ and with as few assumptions as possible on the initial/final conditions.
In order to accomplish this, we will make use of variational methods, which have proved to be very useful in mean field games, both to prove existence of weak solutions \cite{Car15,CarGra,CarGraPorTon,CarMesSan} and also to establish additional regularity \cite{GraMes,ProSan}.
The main idea behind this point of view (which is also exploited in \cite{OrrPorSav}) is that System \eqref{eq:planning_general} can be seen formally as first order necessary optimality conditions of two convex optimization problems in duality, cf.~\cite[Section 2.6]{LasLio07}.

The first problem is a control problem associated to the continuity equation, i.e.
$$\inf_{(m,w)}\int_0^T\int_{\T^d}\left[m H^*(x,-w(t,x)/m(t,x))+F(x,m(t,x))\right]\dd x\dd t$$
subject to $\partial_t m + \nabla\cdot w=0$ and $m(0,\cdot)=m_0$, $m(T,\cdot)=m_T$, where $H^*(x,\cdot)$ is the Fenchel conjugate of $H(x,\cdot)$ and $F(x,\cdot)$ is the antiderivative of $f(x,\cdot)$ w.r.t. the second variable.

The formal dual of this problem can be seen as a control problem associated to the Hamilton-Jacobi equation
$$\inf_{u}\int_0^T\int_{\T^d} F^*(x,-\partial_t u+H(x,\nabla u))\dd x\dd t+\int_{\T^d}u(T,x)\dd m_T(x)-\int_{\T^d}u(0,x)\dd m_0(x),$$ where $F^*(x,\cdot)$ is the Fenchel conjugate of $F(x,\cdot)$.

As the first results of our paper, in Section \ref{sec:well-posed_planning} we show the well-posedness of System \eqref{eq:planning_general} relying on the duality between the previous two convex optimization problems. To show the existence of a solution to the dual problem, we relax it in a suitable way and use a sort of `renormalization trick' that was first used in \cite{CarCarNaz}.

At this point, let us remark that in our existence proof for the dual problem, we require a joint condition on the order of growth of $H$ in the momentum variable and the order of growth of $f$ in the second variable (similarly as in \cite{Car15,CarGra}). This is mainly due to the lack of enough summability on $m$. It is worth mentioning that $L^\infty$ estimates on $m$ would allow us to drop this joint condition on $H$ and $f$. In this context, for the planning problem, in the case of purely quadratic Hamiltonians and $m_0,m_T\in L^\infty(\T^d)$ such $L^\infty$ estimates were obtained by Lions in his lectures (see \cite{Lions_course_planning}). Using completely different techniques, but still in the quadratic Hamiltonian case, such $L^\infty$ estimates on $m$ were obtained recently for mean field games by Lavenant and Santambrogio in \cite{LavSan}. Since in this paper our aim is to consider as general Hamiltonians and initial and final measures as possible, we are not pursuing the higher order summability estimates on $m$.  Such questions would deserve a completely independent study, by themselves. 

Moreover, in contrast to \cite{Car15,CarGra} and \cite{CarGraPorTon}, we also address questions of regularity of weak solutions, based on techniques developed in the recent work \cite{GraMes} by the first two authors.
The inspiration for these results comes from the alternative interpretation of the planning problem in terms of optimal mass transport. Indeed, the variational formulation of both the planning problem and mean field games has its roots in the dynamic formulation of the Monge-Kantorovich optimal transport problem (see \cite{BB}). In the same way, such convex variational problems are underneath other models studying weak solutions to the incompressible Euler equations (see for instance \cite{Bre99, AmbFig2}). The strong convexity present in these problems led Y. Brenier to develop a regularity theory for the pressure field in his model. Inspired by these techniques, the very same ideas were used later successfully to obtain Sobolev regularity for weak solutions of mean field games (see \cite{ProSan,San17,GraMes}). After this series of results it is not unexpected that such results should be obtained for the solutions of the planning problem. This fact also motivates our title, i.e.~the planning problem can be seen as a `regularized' optimal transport problem, where the presence of the coupling and convexity of the Hamiltonian imply immediate Sobolev estimates on the distributional weak solutions. We think that these regularity results in particular could have further impacts on other problems arising in optimal transportation. In this context, let us also mention the very recent paper \cite{LiuLoe} where the authors observe a different (but similar in spirit) regularization effect for very similar optimal transport type problems, in the presence of strong `mean field interaction effects'.

The organization of the paper is as follows.

In Section \ref{sec:well-posed_planning}, we show the well-posedness of System \eqref{eq:planning_general} via the `direct' variational approach, relying on the two convex optimization problems in duality.
We also consider additional regularity of weak solutions of {\color{red}\eqref{eq:planning_general}} in Section \ref{sec:regularity mfg}.
These estimates {\color{red}are} based on the recent work in \cite{GraMes}. Here, unlike in \cite{GraMes} we present also general local in time Sobolev estimates for time derivatives, which complete in some sense the results from \cite{GraMes}. Finally, in Section \ref{sec:open questions} we conclude with some open questions and ideas for further research. 

Let us finish this introduction by summarizing our main results.
Our results on well-posedness for the planning problem are given in Section \ref{sec:well-posed_planning}  and can be informally summarized as follows.

\begin{theorem}[Theorem \ref{theo:solution-equals-optimizer-planning}, Proposition \ref{prop:existence_solution_for_relaxed}]\label{thm:main2}
Let $m_0,m_T\in L^1(\T^d)$ be probability densities. Then under suitable growth and regularity assumptions on $H$ and $f$ (we refer to Section \ref{sec:well-posed_planning} for the precise hypotheses) we have that System \eqref{eq:planning_general} has a weak solution (in the sense of Definition \ref{def:weaksolplanning}). Moreover, if $f$ is strictly increasing in its second variable, then $m$ is unique and if $H$ is strictly convex in the momentum variable, $\nabla u$ is unique on ${\rm{spt}}(m)$.
\end{theorem}

The Sobolev regularity estimates on System \eqref{eq:planning_general} can be informally summarized as follows.

\begin{theorem}[Propositions \ref{prop:space-regularity}-\ref{prop:time-regularity}]\label{thm:main1} 
\vspace{5pt}

(i) Suppose that $m_0, \; m_T \in W^{2,1}(\T^d)$ and suppose further strong convexity and monotonicity conditions on $H$ and $f$ (we refer to Section \ref{sec:regularity mfg} for the details). Then the solution $(u,m)$ of \eqref{eq:planning_general} satisfies 
$$\|m^{\frac{q}{2} - 1}\nabla m\|_{L^2([0,T]\times\T^d)} \leq C,\ \ \ \|m^{1/2}D (j_1(\nabla u))\|_{L^2([0,T]\times\T^d)} \leq C,$$ 
where $(q-1)$  is the order of growth of $f$  $(q>1)$,  and if $H(x,\cdot)\sim|\cdot|^r$ $(r>1)$, then $j_1$ is a function growing like $|\cdot|^{r/2}$. 

\vspace{5pt}

(ii) Under similar assumptions on the data, but no assumptions on $m_0$ and $m_T$,   we have
 $$m^{1/2}\partial_t(j_1(\nabla u))\in L^2_{\rm{loc}}((0,T);L^2(\T^d)) \; \;  \mbox{and } \; \; \partial_t(m^{q/2})\in L^2_{\rm{loc}}((0,T);L^2(\T^d)).$$

\end{theorem} 

\vspace{0.2cm}

{\sc\bf Acknowledgements.} The first author was supported by the National Science Foundation through Grant DMS-1612880. The second author was partially supported by the Air Force under the grant AFOSR MURI FA9550-18-1-0502. The last two authors were partially supported by the  ANR project  MFG ANR-16-CE40-0015-01 and the PEPS-INSMI Jeunes project ``Some open problems in Mean Field
Games'' for the years 2016, 2017 and 2018. This  article also benefited  from the support of the FMJH Program PGMO and from the support of EDF, Thales, Orange et Criteo.  Reference ANR-11-LABX-0056-LMH, LabEx LMH, Project PGMO  VarPDEMFG. The second author would like to thank the hospitality of the Baylor University during his visit in April 2018. The authors thank the anonymous referee for her/his useful comments which helped us to improve the manuscript.

\section{Well-posedness of the planning problem via two convex optimization problems in duality} \label{sec:well-posed_planning}
The general outline for the planning problem is largely the same as for variational mean field games (see e.g. \cite{CarCarNaz,Car15,CarGra,CarGraPorTon}): present two optimization problems in duality, prove that their minimizers are equivalent to solutions to \eqref{eq:planning_general}, and show the existence of minimizers.

\subsection{Assumptions} \label{sec:assms-planning} \vspace{0.3cm}
We assume the following.

\begin{enumerate}
	\item[(H\arabic{hyp})]\label{hipotesi_h_growth} (Conditions on the Hamiltonian I) $H:\bb{T}^d \times \bb{R}^d \to \bb{R}$ is continuous in both variables, convex and differentiable in the second variable $\xi$, with $D_\xi H$ continuous in both variables. Moreover, $H$ has superlinear growth in the gradient variable: there exist $r >1 $ and $C >0$ such that
	\begin{equation}
	\label{eq:hamiltonian_bounds}
	\frac{1}{rC}|\xi|^r-C \leq H(x,\xi) \leq \frac{C}{r}|\xi|^r + C,\ \ \forall\ (x,\xi)\in\T^d\times\R^d.
	\end{equation}
	We denote by $H^*(x,\cdot)$ the Fenchel conjugate of $H(x,\cdot)$, which, due to the above assumptions, satisfies
	\begin{equation}
	\label{eq:hamiltonian_conjugate_bounds}
	\frac{1}{r'C}|\zeta|^{r'}-C \leq H^*(x,\zeta) \leq \frac{C}{r'}|\zeta|^{r'} + C,\ \ \forall\ (x,\zeta)\in\T^d\times\R^d,
	\end{equation}
	where we always use the notation $s' := s/(s-1)$ to denote the conjugate exponent of a number $s \in (1,\infty)$.   \vspace{0.5cm}
	\stepcounter{hyp}
	\item[(H\arabic{hyp})]\label{hipotesi_h_potenza} (Conditions on the Hamiltonian II) The Hamiltonian satisfies 
	\begin{align*}
	H(x,a\xi) \leq aH(x,\xi), \ \ \forall  \;  x\in\T^d, \; \forall \; \xi\in\R^d, \;  \forall a \in [0,1].
	\end{align*}
	A typical Hamiltonian that satisfies (H1) and (H2)  is $H(x,\xi)=b(x)|\xi|^r-c(x)$ for some $b:\T^d\to\R$ continuous positive function and $c:\T^d\to\R$ continuous nonnegative function.

	\stepcounter{hyp} \vspace{0.5cm}
	\item[(H\arabic{hyp})] \label{hyp:coupling} (Conditions on the coupling) The function $f$ is continuous on $\bb{T}^d \times (0,\infty)$, strictly increasing in the second variable. Assume there exist $C>0$ and  $q>1$ such that $r>\max\{d(q-1),1\}$ and  	
	\begin{equation} \label{eq:coupling_growth_f'}
	\frac{1}{C}|m|^{q-1} - C \leq f(x,m) \leq C|m|^{q-1} + C, \hspace{0.3cm} \forall \; m \geq 0,  \; \forall \; x\in\T^d.
	\end{equation}
	With no real loss of generality, we ask for the following normalization:
	\begin{equation}
	\label{eq:normalization}
	f(x,0) =  0 \hspace{0.3cm} \forall  \; x \in \bb{T}^d.
	\end{equation}
	\stepcounter{hyp} \vspace{0.0000001cm}
		\item[(H\arabic{hyp})]  (Conditions on the initial and final measures) The probability measures $m_0$ and $m_T$ are absolutely continuous with respect to $\sL^d\mres\T^d$, with densities still denoted by $m_0$ and $m_T$, respectively. \stepcounter{hyp}
\end{enumerate}\vspace{0.3cm}
We define $F: \T^d \times \RR \to \RR$ so that $F(x,\cdot)$ is an antiderivative of $f(x,\cdot)$ on $(0,\infty)$, that is,
	\begin{equation}
	F(x,m) = \int_0^m f(x,s)\dd s, ~~ \forall ~ m \geq 0.
	\end{equation}
For $m < 0$ we set $F(x,m) = +\infty$.  Note that $F(x,m)\geq 0$ thanks to hypothesis \eqref{eq:normalization}. Moreover, it follows from (H3) that $F$ is continuous on $\bb{T}^d \times [0,\infty)$, for each $x\in \T^d$ the function $F(x,\cdot)$ is strictly convex and differentiable in $(0,+\infty)$, and satisfies the growth condition
	\begin{equation} \label{eq:cost_growth}
	\frac{1}{qC}|m|^q - C \leq F(x,m) \leq \frac{C}{q}|m|^q + C, \hspace{0.3cm} \forall ~ m \geq 0,\ \forall\ x\in\T^d.
	\end{equation}
We define $F^*(x,\cdot):\RR \to \RR$ to be the Fenchel conjugate of $F(x,\cdot)$, i.e. 
$$F^*(x,a) = \sup_{m \geq 0}\left\{ am - F(x,m)\right\}.$$
	Note that $F^{\ast}(x,\cdot)$ is continuous,  increasing and   $F^*(x,a) = 0$ for all $a \leq 0$.
	We also have
	\begin{equation} \label{eq:cost_growth_star}
	\frac{1}{q'C}|a|^{q'} - C \leq F^*(x,a) \leq \frac{C}{q'}|a|^{q'} + C, \hspace{0.3cm} \forall \; a \geq 0,\ \forall \; x\in\T^d.
	\end{equation}

Let us comment on the condition $r > \max\{d(q-1),1\}$ in Hypothesis (H3), since it is the only real growth restriction on the Hamiltonian and coupling terms.
(Note that it is a joint condition: the faster $H$ grows at infinity, the faster $f$ may grow.)
This restriction allows us to prove that a priori $L^\infty$ bounds for a suitable adjoint state $u$, which is a priori merely a subsolution to the Hamilton-Jacobi equation \eqref{eq:planning_general}(i) with no information on initial or final conditions (see the proof of Proposition \ref{prop:existence_solution_for_relaxed}).
To overcome this difficulty, we employ a renormalization trick akin to that found in \cite{CarCarNaz}.
A different strategy is used in \cite{OrrPorSav}, by proving a priori bounds for subsolutions of the Hamilton-Jacobi equation via duality with geodesics in the Wasserstein space.
For this reason their results hold with no growth restriction on the coupling; on the other hand, the Hamiltonian must have quadratic growth.

\subsection{Two optimization problems in duality} \label{sec:2opti-planning}

The planning problem has a variational formulation analogous to what is introduced in \cite{Car15} in the context of MFGs. 


The first optimization problem is described as follows:  let us denote $\mathcal K_0 = \s{C}^1([0,T]\times \T^d)$
and define, on $\mathcal K_0$, the functional 
\begin{multline}
\label{DefmathcalA-planning}
{\mathcal A}(u)= \int_0^T\int_{\T^d}  F^*\left(x,-\partial_tu(t,x)+H(x,\nabla u(t,x)) \right)\dd x\dd t \\
+ \int_{\T^d} u(T,x) m_T(x) \dd x
- \int_{\T^d} u(0,x) m_0(x)\dd x . 
\end{multline}
Notice that,   $F^{\ast}(x,\cdot)$ being increasing and convex, for every $x\in \TT^d$ the function $\RR \times \RR^d \ni (a,b) \mapsto E(x,a,b):=   F^{\ast} \left(x,-a + H(x,b)\right) \in \RR$ is convex and, hence, $\A$ is a convex function. 

The first optimization problem is given by
\begin{equation}
\label{PB:dual2-planning}
\inf_{u\in \mathcal K_0} \mathcal A(u).
\end{equation}
Now, suppose that $(m, w)\in L^1((0,T)\times \T^d) \times L^1((0,T)\times \T^d;\R^d)$ are such that  $m(t,x)\geq 0$ for a.e. $(t,x) \in [0,T] \times \T^d$  and the continuity equation 
\begin{equation}\label{conteq-planning}
\partial_t m+{\rm div} (w)=0\; {\rm in}\; (0,T)\times \T^d,\end{equation}
is satisfied in the distributional sense, i.e. for all $\varphi \in C_c^1((0,T) \times \T^d)$ we have
\be\label{continuity_equation_distributional_sense}
\int_{0}^{T} \int_{\T^d} \left[\partial_t \varphi m + \nabla \varphi \cdot w \right] \dd x \dd t= 0.\ee
Let us denote by $\M(\T^d)$  the space of Radon measures over $\T^d$ and by $\M_+(\T^d)$ the subset of $\M( \T^d)$  given by  the non-negative Radon measures over $\TT^d$. By \cite[Lemma 4.1]{DolNazSav} (see also the discussion in \cite{CarCarNaz}), if \eqref{continuity_equation_distributional_sense} holds, then there exists a unique  weakly-$^\ast$ continuous curve $[0,T] \ni t \mapsto \tilde{m}(t) \in \M_{+}(\TT^d)$ such that for a.e. $t\in [0,T]$ the measure $\tilde{m}(t)$ is absolutely continuous w.r.t. the Lebesgue measure, with density given by  $m(t,\cdot)$, and for all   $\varphi \in C^1([0,T] \times \T^d)$ the following equality holds 
\be\label{continuity_equation_prescribed_initial_and_final_conditions_0}
\int_{t_1}^{t_2} \int_{\T^d} \left[\partial_t \varphi m + \nabla \varphi \cdot w \right] \dd x \dd t= \int_{\T^d} \varphi(t_2,x) \dd \tilde{m}(t_2)(x)\dd x -	 \int_{\T^d} \varphi(t_1,x) \dd \tilde{m}(t_1)(x),
\ee
for all $0 \leq t_1 < t_2 \leq T$. 

Define   $\mathcal K_1$ as the set of pairs $(m, w)\in L^1((0,T)\times \T^d) \times L^1((0,T)\times \T^d;\R^d)$ such that $m(t,x)\geq 0$ for a.e. $(t,x) \in [0,T] \times \T^d$, equation \eqref{continuity_equation_distributional_sense} is satisfied in the distributional sense and $\tilde{m}(0)$ and $\tilde{m}(T)$ are absolutely continuous with respect to the Lebesgue measure with densities given by $m_0$ and $m_T$, respectively. Note that \eqref{continuity_equation_prescribed_initial_and_final_conditions_0} implies that the last two requirements are equivalent to the fact that $(m,w)$ satisfies  
\be\label{continuity_equation_prescribed_initial_and_final_conditions}
\int_{0}^{T} \int_{\T^d} \left[\partial_t \varphi m + \nabla \varphi \cdot w \right] \dd x \dd t= \int_{\T^d} \varphi(T,x) m_T(x)\dd x -	 \int_{\T^d} \varphi(0,x) m_0(x)\dd x,
\ee
for all $\varphi \in C^1([0,T] \times \T^d)$. Notice that if $(m,w) \in \K_1$, then $\int_{\TT^d}m_0(x) \dd x= 1$ and \eqref{continuity_equation_prescribed_initial_and_final_conditions_0}  imply that $\int_{\TT^d}m(t,x) \dd x= 1$ for a.e. $t\in [0,T]$. 
%
On $\mathcal K_1$, let us  define 
$$
{\mathcal B}(m,w):=  \int_0^T\int_{\T^d} \left[m(t,x) H^*\left(x, -\frac{w(t,x)}{m(t,x)}\right)+ F\left(x,m(t,x)\right)\right] \dd x\dd t,
$$
where, for $a=0$ and $b\in \RR^d$, we impose that 
$$
a  H^*\left(x, -\frac{b}{a}\right)=\left\{\begin{array}{ll}
+\infty & {\rm if }\; b \neq 0,\\
0 & {\rm if }\; b=0.
\end{array}\right.
$$
Under this definition, it is easy to check that for every $x\in \TT^d$, the function  $\RR_+ \times \RR^d \ni (a,b) \mapsto  a  H^*\left(x, -\frac{b}{a}\right) \in \RR \cup \{+\infty\}$ is proper, convex and lower-semicontinuous and, hence, $\B$ is convex.

The second optimization problem is the following: 
\begin{equation}
\label{Pb:mw2-planning}
\inf_{(m,w)\in \mathcal K_1} \mathcal B(m,w)\;.
\end{equation}
The following lemma is proved using a similar argumentation as in \cite[Proposition 2.1]{CarCarNaz} and \cite[Lemma 2]{Car15}. For the sake of completeness, we provide the details of the proof.

\begin{lemma}\label{Lem:dualite-planning} We have
\begin{equation}
\label{no_duality_gap}
\inf_{u\in \mathcal K_0}{\mathcal A}(u) = - \min_{(m,w)\in \mathcal K_1} {\mathcal B}(m,w). 
\end{equation}
Moreover, there exists a unique $(\bar m, \bar w) \in \K_1$ such that $\ds{\mathcal B}(\bar m,\bar w)=  \min_{(m,w)\in \mathcal K_1} {\mathcal B}(m,w)$. Setting $\ell:= \frac{r'q}{r'+q-1}>1$, this minimizer also satisfies  $(\bar m,\bar w) \in  L^q((0,T)\times \T^d)\times L^{\ell}((0,T)\times \T^d;\R^d)$ and
\begin{equation}
\label{additional_regularity_solution_opt_problem}
\|\bar m\|_{L^q}  + \| \bar w \|_{L^\ell} \leq C,
\end{equation}
where $C>0$ is a constant independent of $m_0$ and $m_T$.
\end{lemma}
\begin{proof}
Let $\mathcal H_0:=\s{C}^0([0,T]\times \T^d) \times \s{C}^0([0,T]\times \T^d; \RR^d)$ and define the bounded linear operator $\Lambda:  \K_0 \to\mathcal H_0$ by $\Lambda u= (\partial_t u, \nabla u)$, and the functionals $J_1:\mathcal H_0 \to \RR$, $J_2:  \K_0 \to \RR$, respectively, by
$$\ba{rcl}
\ds J_1 \left(\phi_1, \phi_2 \right)&=& \ds\int_{0}^{T} \int_{\T^d} E \left(x,\phi_1(x), \phi_2(x)\right) \dd x \dd t, \\[8pt]
\ds J_2 (u)&=&  \ds\int_{\TT^d} u(T,x) m_T(x) \dd x -  \int_{\TT^d} u(0,x) m_0(x) \dd x,
\ea$$
where we recall that $ E(x,a,b):= F^{\ast}(x,-a + H(x,b))$ for all $x\in \T^d$, $a \in \RR$ and $b \in \RR^d$. 
Thus, problem \eqref{PB:dual2-planning} can be rewritten as 
\begin{equation}
\label{rewritting_cost_functional}
\inf_{u \in \K_0} \left\{J_1\left( \Lambda u\right) +J_2(u)\right\}.
\end{equation}
Since 
$$
J_1\left( \Lambda\left(u -\min_{x \in \TT^d} u(T,x)\right)\right)+ J_2\left(u -\min_{x \in \TT^d} u(T,x)\right)= J_1(\Lambda u) + J_2(u), \hspace{0.3cm} \forall \; u \in \K_0,
$$
we can assume that the infimum in \eqref{rewritting_cost_functional} is taken over  $u\in \K_0$ such that $\inf_{x\in \TT^d} u(T,x)=0$. Using this fact, \eqref{eq:cost_growth_star}, estimate \eqref{inequality_with_infimum} (proved  after Lemma \ref{one_holder_bound} below), and setting $\corrd{\bar a}:= \| -\partial_tu +H(\cdot,\nabla u) \|_{L^{q'}}$, we get the existence of $c_1>0$, $c_2\in \RR$ and $c_3 \in \RR$ (independent of $m_0$ and $m_T$) such that 
\begin{equation}
\label{definition_underline_c}
J_1\left( \Lambda u\right) +J_2(u) \geq     c_1 \bar a^{q'} + c_2 \bar a  +c_3 \geq \underline{c}:=\inf_{\tau \in \RR_+} \left\{c_1 \tau^{q'} + c_2 \tau\right\}  +c_3 > -\infty.
\end{equation}
This proves that the infimum in \eqref{rewritting_cost_functional} is finite. Using that $J_1$ and $J_2$ are continuous, by the Fenchel-Rockafellar theorem (see e.g. \cite[Chapter 3, Theorem 4.1]{ekeland-temam}) we have that
\begin{equation}
\label{opitmization_with_Js}
\inf_{u\in \mathcal K_0}{\mathcal A}(u)=  - \min \left\{ J_1^{\ast}(-(m,w)) + J_{2}^{\ast}(\Lambda^{\ast}(m,w))  | \; (m,w)\in \mathcal H_0^*\; \right\},
\end{equation}
where $\mathcal H_0^*=\M((0,T)\times \TT^d)) \times \M((0,T)\times \TT^d)^d $.
 Let us provide a more explicit expression of the right hand side above.  By \cite[Theorem 5]{Roc71}, we have that 
$$
J_1^{\ast}(m,w)= \int_{0}^{T} \int_{\T^d} E^{\ast}(x,m^{ac},w^{ac}) \dd x \dd t +  \int_{0}^{T} \int_{\T^d}  E_{\infty}^{\ast}\left(x,\frac{\dd m^{s}}{\dd \theta},\frac{\dd w^{s}}{\dd \theta}\right)\dd \theta (t,x),
$$
where $(m^{ac},w^{ac})$ and $(m^{s},w^{s})$ denote, respectively, the absolutely continuous and singular parts of $(m,w)$ w.r.t. the Lebesgue measure, $\theta \in \M((0,T) \times \TT^d)$ is any Radon measure such that $(m^s,w^s)$ is absolutely continuous w.r.t. $\theta$, and $E_{\infty}^{\ast}(x,\cdot,\cdot)$ is the recession function of $E^{\ast}(x,\cdot,\cdot)$. We easily check that 
\begin{equation}
\label{computation_E_ast}
E^{\ast}(x,m,w)= \left\{ \ba{ll}
- m H^{\ast}\left(x,-\frac{w}{m} \right) + F(-m),  & \mbox{if } m<0, \\[6pt]
0,  & \mbox{if } (m,w)=(0,0),\\[6pt]
+\infty, & \mbox{otherwise}.\ea \right.
\end{equation}
Since $E^{\ast}(x,0,0)=0 < +\infty$,  the recession function can be computed as follows
$$E_{\infty}^{\ast}\left(x,h_m,h_w\right)= \lim_{\lambda \to +\infty} \frac{E^{\ast}(x,\lambda h_m, \lambda h_w)}{\lambda}=  \left\{ \ba{ll} 0, & \mbox{if } (h_m,h_w)=(0,0), \\[6pt]
+ \infty,  & \mbox{otherwise.}\ea\right.
$$
In the second equality above, we have used \eqref{eq:cost_growth_star}. We deduce that if $(m,w) \notin  L^1((0,T)\times \T^d) \times L^1((0,T)\times \T^d;\R^d)$, then $J_1^{\ast}(m,w) = +\infty$. If $(m,w) \in  L^1((0,T)\times \T^d) \times L^1((0,T)\times \T^d;\R^d)$, then 
$$
J_{2}^{\ast}(\Lambda^{\ast}(m,w))=\sup_{u \in \K_0} \left\{\langle   \Lambda^{\ast}(m,w),u \rangle_{\K_0^{\ast},\K_0} - \int_{\TT^d} u(T,x) m_T(x) \dd x + \int_{\TT^d} u(0,x) m_0(x) \dd x\right\},
$$
where $\langle \cdot, \cdot \rangle_{\K_0^{\ast},\K_0}$ denotes the duality product between $\K_0$ and $\K_0^{\ast}$. 
Using that 
$$
\langle   \Lambda^{\ast}(m,w),u \rangle_{\K_0^{\ast},\K_0}= \langle   (m,w),\Lambda u \rangle_{\mathcal H_0^*,\mathcal H_0}= \int_{0}^{T} \int_{\T^d} \left[  m\partial_t u +  w \cdot \nabla u \right] \dd x \dd t,
$$
we get that $J_{2}^{\ast}(\Lambda^{\ast}(m,w))<+\infty$ if and only if $J_{2}^{\ast}(\Lambda^{\ast}(m,w))=0$, which is equivalent to the fact that  $(m,w)$ satisfies \eqref{continuity_equation_prescribed_initial_and_final_conditions}. We conclude that  the optimization problem in the r.h.s. of \eqref{opitmization_with_Js} admits a solution $(\bar{m},\bar{w})$, is equivalent to problem \eqref{Pb:mw2-planning} and, hence, \eqref{no_duality_gap} holds true. 

By \eqref{definition_underline_c}, we have  $\mathcal B(\bar{m},\bar{w}) \leq -\underline{c}$ with  $\underline{c}$ independent of $m_0$ and $m_T$. Using this estimate and  arguing as in  the proof of  \cite[Lemma 2]{Car15}, we easily obtain \eqref{additional_regularity_solution_opt_problem}.  Finally, the uniqueness of the solution $(\bar m,\bar w)$ to  \eqref{Pb:mw2-planning} follows exactly as in the proof of  \cite[Lemma 2]{Car15}.
\end{proof}
\begin{remark} The previous proof shows that the results in Lemma \ref{Lem:dualite-planning} are  valid also when $m_0$ and $m_T$ belong  to $\sP(\TT^d)$, without any summability assumptions. 
\end{remark}
Now, we consider a relaxation of Problem \eqref{PB:dual2-planning}.
Let ${\mathcal K}$ be the set of pairs $(u,\alpha)\in BV((0,T)\times \T^d)\times  L^{q'}((0,T)\times \T^d)$ such that    $\nabla u\in L^r((0,T)\times \T^d;\R^d)$, $u \in L^{\infty}((0,T)\times \T^d)$, the traces $u(0,\cdot)$, $u(T,\cdot)$ of $u$ on $\{0\} \times \TT^d$ and $\{T\} \times \TT^d$ (see e.g. \cite[Section 3.8]{AmbFusPal}), respectively, belong to $L^{\infty}(\TT^d)$,   and 
\begin{equation*} 
-\partial_t u + H(x,\nabla u)\le \a
\end{equation*}
holds in the sense of distributions on $(0,T) \times \T^d.$

We extend the functional $\mathcal A$ to  ${\mathcal K}$ by setting 
$$
\mathcal A(u,\alpha) = \int_0^T\int_{\T^d}  F^*(x,\alpha(t,x)) \dd x\dd t
+ \int_{\T^d} u(T,x)m_T(x) \dd x - \int_{\T^d} u(0,x)m_0(x) \dd x \hspace{0.3cm} \forall \; (u, \alpha)\in {\mathcal K}.
$$

We consider  the following relaxation of Problem \eqref{PB:dual2-planning}:
\begin{equation}
\label{PB:dual-relaxed-plan}
\inf_{(u, \alpha)\in {\mathcal K}} \mathcal A(u,\alpha)
\end{equation}

\begin{proposition}\label{prop:weakequalsstrong-plan} We have
	\begin{equation}
	\label{equality_relaxed_non_relaxed_problems}
	\inf_{u\in \mathcal K_0} \mathcal A(u)= \inf_{(u, \alpha)\in {\mathcal K}} \mathcal A(u,\alpha)= -\min_{(m,w) \in \K_1} \B(m,w).
	\end{equation}
\end{proposition}

The proof of Proposition \ref{prop:weakequalsstrong-plan} follows easily once we have the following technical lemma.
\begin{lemma}\label{lem:phialphamw-plan} Let $(u,\alpha)\in {\mathcal K}$ and $(m,w)\in {\mathcal K}_1$. Assume that $mH^*(\cdot, -w/m)\in L^1((0,T)\times \T^d)$, $m \in L^q((0,T) \times \bb{T}^d)$, and $m_0$, $m_T  \in L^1(\bb{T}^d)$. Then $\alpha m \in L^1((0,T) \times \bb{T}^d)$, and for almost all $t \in (0,T)$ we have
	\begin{equation}
	\label{ineq:ineqfirst-plan}
	\int_{\T^d} (u(T) m_T -  u(t)m(t)) \dd x  + \int_t^{T}\int_{\T^d} m \left[\alpha  + H^*\left(x,-\frac{w}{m}\right)\right]\dd x \dd t \; \geq \; 0,
	\end{equation}
	and 
	\begin{equation}
	\label{ineq:ineqfirst-plan_i}
	\int_{\T^d} (u(t)m(t) - u(0)m_0) \dd x + \int_0^{t}\int_{\T^d} m \left[\alpha  + H^*\left(x,-\frac{w}{m}\right)\right]\dd x \dd t \; \geq \; 0.
	\end{equation}
	Moreover, if equality holds in the inequality \eqref{ineq:ineqfirst-plan} for $t=0$, then $w= -mD_\xi H(\cdot,\nabla u)$ a.e.~and $-\partial_t u^{\rm{ac}}(t,x) + H(x,\nabla u(t,x)) = \alpha(t,x)$ for $m$-a.e.~$(t,x) \in (0,T) \times \bb{T}^d$, where $\partial_t u^{\rm{ac}}$ is the absolutely continuous part of the measure $\partial_t u$.
\end{lemma}

\begin{proof}
	The proof is an adaptation of the argument seen in 
\cite[Lemma 2.4]{CarGra}.
	We will prove \eqref{ineq:ineqfirst-plan}; the proof of \eqref{ineq:ineqfirst-plan_i} is analogous.
	\begin{uncomment}
	{\color{blue}{\bf Be careful, in the proof of \cite[Lemma 2.4]{CarGra} a continuity assumption on $m_0$ is used. We should try to extend the result to the case stated above using the regularization technique of the analogous proof in the Appendix for dealing with $\beta$, or, alternatively, arguing exactly as in the proof of \cite[Lemma 2.4]{CarGra} but using Lusin theorem to pass from $L^1$ to $C$ on a compact set with almost complete measure contained in $\T^d$. }}
	\end{uncomment}

	We first extend $(m,w)$ to $\bb{R}\times \bb{T}^d$ by setting $(m,w)=(m_0,0)$ on $(-\infty,0)\times \bb{T}^d$ and $(m,w)= (m(T),0)$ on $(T, +\infty)\times \bb{T}^d$. Note that we still have $\partial_t m+ {\mathrm{div}}~w = 0$ on $\bb{R}\times \bb{T}^d$.
	For $\e>0$, let $\xi_\e(t,x)$ be a sequence of smooth convolution kernels that we will define below. Define $m_\e := \xi_\e \ast m$ and $w_\e := \xi_\e \ast w$. Then $m_\e$ and $w_\e$ are $C^\infty$ smooth, $m_\e > 0$, and
	\begin{equation}
	\partial_t m_\e + {\mathrm{div}}~w_\e = 0.
	\end{equation}
	Recalling that $-\partial_t u + H(x,\nabla u) \leq \alpha$ in the sense of distributions, we deduce
	\begin{multline} \label{eq:ineq eps}
	\int_{\bb{T}^d} u(t)m_\e(t)\dd x - \int_{\bb{T}^d} u(T)m_\e(T)\dd x
	\leq  \int_t^T \int_{\bb{T}^d} \left[-w_\e \cdot \nabla u - m_\e H(x,\nabla u) + m_\e \alpha\right]\dd x\dd t\\
	\leq  \int_t^T \int_{\bb{T}^d} \left[m_\e H^*\left(x,-\frac{w_\e}{m_\e}\right) + m_\e \alpha\right]\dd x\dd t
	\end{multline}
	for any $t \in (0,T)$.
	As $\e \to 0$, we have that $m_\e \to m$ in $L^q((0,T) \times \bb{T}^d)$, and in particular $m_\e(t) \to m(t)$ in $L^q(\bb{T}^d)$ for almost every $t \in (0,T)$, while $m_\e \alpha \to m \alpha$ in $L^1((0,T)\times\T^d)$ since $\alpha \in L^{q'}((0,T)\times\T^d)$. Thus as $u \in L^\infty((0,T) \times \bb{T}^d)$, we get $\int_{\T^d} u(t)m_\e(t)\dd x \to \int_{\T^d} u(t)m(t)\dd x$ for almost every $t \in (0,T)$.
	On the other hand, by the argument given in \cite[Lemma 2.4]{CarGra}, we have
	\begin{equation} \label{eq:limit_L}
	\lim_{\e \to 0}  \int_t^T \int_{\bb{T}^d} m_\e H^*\left(x,-\frac{w_\e}{m_\e}\right)\dd x\dd t =  \int_t^T \int_{\bb{T}^d} m H^*\left(x,-\frac{w}{m}\right)\dd x\dd t.
	\end{equation}
	Then
	\begin{equation}
	\int_{\bb{T}^d} u(t)m(t)\dd x \leq \limsup_{\e \to 0} \int_{\bb{T}^d} u(T)m_\e(T)\dd x + \int_t^T \int_{\bb{T}^d} \left[m H^*\left(x,-\frac{w}{m}\right) + m \alpha\right]\dd x\dd t.
	\end{equation}
	
	To conclude, we just need to show  that 
	\begin{equation}
	\int_{\bb{T}^d} u(T) m_{\e}(T)\dd x \to \int_{\bb{T}^d} u(T) m_T\dd x, ~~ \e \da 0.
	\end{equation}
	Since $u(T) \in L^\infty(\bb{T}^d)$, it is enough to show $m_{\e}(T) \to m_T$ in $L^1(\bb{T}^d)$ as $\e\da 0$.
	For this we choose a particular construction of the convolution kernel $\xi_\e$.
	
	Let $\eta : \bb{R} \to (0,\infty)$ and $\psi : \bb{R}^d \to (0,\infty)$ be even convolution kernels, each with compact support in the unit ball, $\d>0$ and set $\eta_\delta(t) = \delta^{-1}\eta(t/\delta)$ and $\psi_\e(x) = \e^{-d}\psi(x/\e)$. We will choose $\xi_{\e}(t,x) = \eta_\delta(t)\psi_\e(x)$ where $\delta = \delta(\e)$ will be determined by the following calculations.
	Set $m_{T,\e} = \xi_\e \ast m_T$.
	Our first observation is that
	\begin{multline} \label{eq:m ep to m}
	\int_{\bb{T}^d}|m_\e(T,x) - m_T(x)|\dd x
	\leq \int_{\bb{T}^d}\left|\int_{T-\delta}^{T+\delta} \int_{\bb{T}^d}\eta_\delta(T-s)\psi_\e(x-y)(m(s,y)-m_T(y)) \dd y \dd s\right|\dd x\\
	+ \int_{\bb{T}^d}\left|\int_{T-\delta}^{T+\delta}\int_{\bb{T}^d}\eta_\delta(T-s)\psi_\e(x-y)(m_T(y)-m_T(x)) \dd y \dd s\right|\dd x\\
	= \int_{\bb{T}^d}\left|\int_{T-\delta}^{T+\delta} \int_{\bb{T}^d}\int_s^T \eta_\delta(T-s)\nabla \psi_\e(x-y) \cdot w(\tau,y) \dd y \dd \tau \dd s  \right|\dd x
	+ \int_{\bb{T}^d}\left|m_{T,\e}(x)-m_T(x)\right|\dd x\\
	\leq \frac{C}{\e^{d+1}}\int_{T-\delta}^T \int_{\bb{T}^d} |w(\tau,y)|\dd \tau \dd y + \int_{\bb{T}^d}\left|m_{T,\e}(x)-m_T(x)\right|\dd x.
	\end{multline}
	We set $\delta = \delta(\e)$ small enough such that $\frac{C}{\e^{d+1}}\int_{T-\delta}^T \int_{\bb{T}^d} |w(\tau,y)|\dd \tau \dd y \leq \e$.
	Then \eqref{eq:m ep to m} proves that $m_\e(T,\cdot) \to m_T$ in $L^1$ as $\e \to 0$.
	The proof of \eqref{ineq:ineqfirst-plan} is complete.

\end{proof}
\begin{proof}[Proof of Proposition \ref{prop:weakequalsstrong-plan} ] 
Fixing $t\in (0,T)$ such that \eqref{ineq:ineqfirst-plan} and \eqref{ineq:ineqfirst-plan_i} hold, by adding both inequalities we get that 
	\begin{equation}
	\label{ineq:ineqfirst-plan_ii}
	\int_{\T^d} (u(T) m_T -  u(0)m(0)) \dd x  + \int_0^{T}\int_{\T^d} m \left[\alpha  + H^*\left(x,-\frac{w}{m}\right)\right]\dd x \dd t \; \geq \; 0,
	\end{equation}
for every $(u,\alpha)\in \K$ and $(m,w) \in \K_1$ satisfying the assumptions of Lemma \ref{lem:phialphamw-plan}. Thus, 
$$\ds\ba{rcl}
\A(u,\alpha) &\geq& -\ds \int_0^{T}\int_{\T^d} \left[ m \left(\alpha  + H^*\left(x,-\frac{w}{m}\right)\right) - F^{\ast}(x,\alpha)\right]\dd x \dd t \\[6pt]
   \; & \geq  & \ds-\int_0^{T}\int_{\T^d} \left[H^*\left(x,-\frac{w}{m} \right)m  +F(x,m) \right] \dd x \dd t,\ea
$$
from which we deduce that $\inf_{(u,\alpha) \in \K} \A(u,\alpha)  \geq - \min_{(m,w) \in \K_1}\B(m,w) $.  Therefore, \eqref{equality_relaxed_non_relaxed_problems} follows from the inequalities
$$
- \min_{(m,w) \in \K_1}\B(m,w)= \inf_{ u \in \K_0} \A(u,\alpha) \geq \inf_{(u,\alpha) \in \K} \A(u,\alpha)   \geq - \min_{(m,w) \in \K_1}\B(m,w).
$$
\end{proof}
\subsection{Weak solutions and minimizers} \label{sec:weak solutions planning}

The definition of weak solutions for the planning problem is analogous to that of the mean field game system (see \cite{Car15,CarGra}).
\begin{definition}\label{def:weaksolplanning} Let $(u,m)\in BV((0,T)\times\T^d) \times L^q((0,T)\times\T^d)$. 
	We say that $(u,m)$ is a weak solution to \eqref{eq:planning_general} if 
	\begin{itemize}
		\item[(i)] the following integrability conditions hold:
		$$ \ba{l}
		\ds \nabla u\in L^r((0,T)\times\T^d;\R^d), \hspace{0.3cm} u \in L^\infty((0,T)\times\T^d),  \\[6pt]	
		\mbox{the traces } u(0,\cdot), \;  u(T,\cdot) \; \; \mbox{belong to } L^{\infty}(\T^d),\\[6pt]	  
		\ds mH^*(\cdot, D_\xi H(\cdot,\nabla u))\in L^1((0,T)\times\T^d), \hspace{0.3cm} m D_\xi H(\cdot,\nabla u))\in L^1((0,T)\times\T^d;\R^d).
	    \ea$$
		
		\item[(ii)] Equation \eqref{eq:planning_general}-{\rm(i)} holds in the following sense: inequality
		\begin{equation}
		\label{eq:distrib-plan}
		\ds \quad -\partial_t u+H(x,\nabla u)\leq  f(x,m) \quad {\rm in }\; (0,T)\times \T^d
		\end{equation}
		holds in the sense of distributions.  \vspace{0.2cm}
		
		\item[(iii)] Equation \eqref{eq:planning_general}-{\rm(ii)} holds:  
		\begin{equation}
		\label{eqcontdef-plan}
		\ds \quad \partial_t m-{\rm div}(m D_\xi H(x,\nabla u)))= 0\  {\rm in }\; (0,T)\times \T^d, \quad m(0)=m_0, \ m(T) = m_T
		\end{equation}
		in the weak sense \eqref{continuity_equation_prescribed_initial_and_final_conditions}; and, finally, \vspace{0.2cm}
		
		\item[(iv)] the following equality holds: 
		\begin{equation}
		\label{defcondsup-plan} \begin{array}{l}
		\ds \int_0^T\int_{\bb{T}^d} m(t,x)\left[f(x,m(t,x))+ H^*(x, D_\xi H(x,\nabla u)(t,x)) \right]\dd x\dd t\\
		\ds\qquad \qquad \qquad \qquad \qquad \qquad \qquad + \int_{\bb{T}^d}\left[ m_T(x)u(T,x)-m_0(x)u(0,x)\right]\dd x=0. 
		\end{array}
		\end{equation}
	\end{itemize}
\end{definition}

To prove the existence of weak solutions, we will use the fact that they are equivalent to minimizers of the two optimization problems presented in Section \ref{sec:2opti-planning}.
In Section \ref{section_existence_minimizers_planning_problem} below we will show the existence of minimizers for problem \eqref{PB:dual-relaxed-plan}, and, hence, the existence of solutions to \eqref{eq:planning_general}.

\begin{theorem}\label{theo:solution-equals-optimizer-planning}
	Let $(\bar m,\bar w)\in \mathcal K_1$ be a minimizer of \eqref{Pb:mw2-planning} and $(\bar u,\bar \alpha)\in \mathcal K$ be a minimizer of \eqref{PB:dual-relaxed-plan}. Then $(\bar u, \bar m)$ is a weak solution of the planning problem \eqref{eq:planning_general} and $\bar w= -\bar mD_\xi H(\cdot,\nabla \bar u)$, while $\bar\alpha= f(\cdot,\bar m)$ a.e.
	
	Conversely, any weak solution $(\bar u,\bar m)$ of \eqref{eq:planning_general} is such that the pair $(\bar m,-\bar mD_\xi H(\cdot,\nabla \bar u))$ is the minimizer of \eqref{Pb:mw2-planning} while $(\bar u, f(\cdot,\bar m))$ is a minimizer of \eqref{PB:dual-relaxed-plan}. 
\end{theorem}

\begin{proof}
	Let $(\bar m,\bar w) \in \s{K}_1$ be a minimizer of Problem \eqref{Pb:mw2-planning} and $(\bar u,\bar \alpha,) \in \s{K}$ be a minimizer of Problem  \eqref{PB:dual-relaxed-plan}.
	Due to  Proposition \ref{prop:weakequalsstrong-plan}, we have
	\begin{equation*}
	\int_0^T \int_{\bb{T}^d} \left[F^*(x,\bar \alpha) + F(x,\bar m) + \bar mH^*\left(x,-\frac{\bar w}{\bar m}\right)\right] \dd x\dd t + \int_{\bb{T}^d}\left[ \bar u(T)m_T - \bar u(0)m_0 \right]\dd x = 0.
	\end{equation*}
	We show that $\bar \alpha = f(x,\bar m)$. Indeed, by the definition of Legendre transform,
	\begin{equation} \label{eq:F_Fstar_geq plan}
	F^*(x,\bar \alpha(t,x)) + F(x,\bar m(t,x)) \geq \bar \alpha(t,x)\bar m(t,x), 
	\end{equation}
	hence
	$$
	\int_0^T \int_{\bb{T}^d}  \left[\bar \alpha(t,x)\bar m(t,x)+ \bar mH^*\left(x,-\frac{\bar w}{\bar m}\right)\right] \dd x \dd t + \int_{\bb{T}^d} \left[\bar u(T) m_T -\bar  u(0)m_0\right] \dd x \leq 0.
	$$
	Thanks to Lemma \ref{lem:phialphamw-plan}, the above inequality is in fact  an equality,  $\bar w=-\bar mD_\xi H(\cdot,\nabla \bar u)$ a.e.~and Equation \eqref{eq:F_Fstar_geq plan} becomes equality a.e. Therefore, by the convexity and differentiability of $F$,
	\begin{equation} \label{eq:alpha_equals_f plan}
	\bar \alpha(t,x) = f(x,\bar m(t,x))
	\end{equation}
	almost everywhere and \eqref{defcondsup-plan} holds for $(\bar u, \bar m)$.
	In particular, $\bar mH^*(\cdot, D_\xi H(\cdot,\nabla \bar u))\in L^1((0,T)\times\T^d)$.	
	Moreover, since $(\bar u,\bar \alpha) \in \s{K}$ and Equation \eqref{eq:alpha_equals_f plan} holds, we have $-\partial_t\bar  u + H(x,\nabla \bar u) \leq f(x,\bar m)$ in the sense of distributions.
	Furthermore, since $(\bar u,\bar \alpha) \in \s{K}$  and $\bar w=-\bar mD_\xi H(\cdot,\nabla \bar u)$, we have that $\bar mD_\xi H(\cdot,\nabla \bar u)\in L^1((0,T)\times\T^d;\R^d)$ and \eqref{eqcontdef-plan} holds in the sense of distributions.
	Therefore $(\bar u, \bar m)$ is a solution in the sense of Definition \ref{def:weaksolplanning}.
	
	Suppose now that $(\bar u,\bar m)$ is a weak solution of \eqref{eq:planning_general} as in Definition \ref{def:weaksolplanning}. Set $\bar w = -\bar mD_\xi  H(\cdot,\nabla \bar u)$, $\bar \alpha(t,x) = f(x,\bar m(t,x))$. By definition of weak solution $(\bar w,\bar \alpha) \in L^1((0,T)\times\T^d;\R^d)\times L^1((0,T)\times\T^d)$, $\bar m\in L^q((0,T)\times\T^d)$, and $\bar u\in L^\infty((0,T)\times\T^d)$.  Moreover, since  $\bar m\in L^q((0,T)\times\T^d)$, the growth condition \eqref{eq:coupling_growth_f'} implies that $\bar \alpha \in L^{q'}((0,T)\times\T^d)$.
	Therefore $(\bar m,\bar w) \in \s{K}_1$ and $(\bar u,\bar \alpha) \in \s{K}$.
	
	It remains to show that $(\bar u,\bar \alpha)$ minimizes $\s{A}$ and $(\bar m,\bar w)$ minimizes $\s{B}$.	
	Let $(\bar u',\bar \alpha') \in \s{K}$. By the convexity and differentiability of $F^*$ in the second variable, we have
	\begin{align*}
	\s{A}(\bar u',\bar \alpha') &= \int_0^T \int_{\bb{T}^d} F^*(x,\bar \alpha'(t,x))\dd x \dd t + \int_{\bb{T}^d}\left[ \bar u'(T,x)m_T(x) - \bar u'(0,x)m_0(x)\right]\dd x\\
	&\geq \int_0^T \int_{\bb{T}^d} \left[F^*(x,\bar \alpha(t,x)) + \partial_\alpha F^*(x,\bar \alpha(t,x))(\bar \alpha'(t,x) - \bar \alpha(t,x)) \right]\dd x \dd t \\
	& \quad \quad \quad + \int_{\bb{T}^d} \left[ \bar u'(T,x)m_T(x) - \bar u'(0,x)m_0(x)\right]\dd x\\
	&= \int_0^T \int_{\bb{T}^d} \left[F^*(x,\bar \alpha(t,x)) +\bar  m(t,x)(\bar \alpha'(t,x) - \bar \alpha(t,x))\right] \dd x \dd t \\
	& \quad \quad \quad + \int_{\bb{T}^d} \left[\bar u'(T,x)m_T(x) - \bar u'(0,x)m_0(x)\right]\dd x\\
	&= \s{A}(\bar u,\bar \alpha) + \int_0^T \int_{\bb{T}^d}  \bar m(t,x)(\bar \alpha'(t,x) -\bar  \alpha(t,x)) \dd x \dd t\\
	&\quad \quad \quad  + \int_{\bb{T}^d} (\bar u'(T,x)-\bar u(T,x))m_T(x)\dd x + \int_{\bb{T}^d}(\bar u(0,x)-\bar  u'(0,x))m_0(x)\dd x\\
	&= \s{A}(\bar u,\bar \alpha) + \int_0^T \int_{\bb{T}^d} \left[ \bar m(t,x)\bar \alpha'(t,x) + \bar m(t,x)H^*\left(x,-\frac{\bar w(t,x)}{\bar m(t,x)}\right)\right]\dd x \dd t\\
	&\quad \quad \quad  + \int_{\bb{T}^d} \bar u'(T,x)m_T(x)\dd x - \int_{\bb{T}^d}\bar  u'(0,x)m_0(x)\dd x
	\end{align*}
	where the last equality follows from Equation \eqref{defcondsup-plan}.
	Applying Lemma \ref{lem:phialphamw-plan} applied to $(\bar u',\bar \alpha')$  and $(\bar m,\bar w)$, 
	we deduce
	$$
	\s{A}(\bar u',\bar \alpha') \geq \s{A}(\bar u,\bar \alpha),
	$$
	and so $(\bar u,\bar \alpha)$ is a minimizer of $\s{A}$. 
	
	The argument for $(\bar m,\bar w)$ is similar. Let $(\bar m',\bar w')$ be a competitor for $\s{B}$. Then because $F$ is convex and differentiable in the second variable, we have, using Equation \eqref{defcondsup-plan},
	\begin{align*}
	\s{B}(\bar m',\bar w') &= \int_0^T\int_{\T^d} \left[\bar m'H^*\left(x,-\frac{\bar w'}{\bar m'}\right) + F(x,\bar m')\right]\dd x\dd t\\
	&\geq \int_0^T\int_{\T^d} \left[ \bar m'H^*\left(x,-\frac{\bar w'}{\bar m'}\right) + F(x,\bar m) + f(x,\bar m)(\bar m'-\bar m)\right]\dd x\dd t\\
	&= \int_{\bb{T}^d} \left[\bar u(T) m_T-\bar u(0)m_0\right]\dd x 
	\\
	& \quad \quad \quad + \int_0^T\int_{\T^d} \left[\bar m'H^*\left(x,-\frac{\bar w'}{\bar m'}\right) + \bar mH^*\left(x,-\frac{\bar w}{\bar m}\right) + F(x,\bar m) + \bar \alpha\bar m'\right]\dd x\dd t\\
	&=\s{B}(\bar m,\bar w)+ \int_{\bb{T}^d} \left[\bar u(T) m_T-\bar u(0)m_0\right]\dd x + \int_0^T\int_{\T^d}\left[ \bar m'H^*\left(x,-\frac{\bar w'}{\bar m'}\right)  + \bar \alpha \bar m'\right]\dd x\dd t\\
	&\geq \s{B}(\bar m,\bar w).
	\end{align*}
	Here  we applied  Lemma \ref{lem:phialphamw-plan} to $(\bar u,\bar \alpha)$  and $(\bar m',\bar w')$ in the last line. Therefore $(\bar m,\bar w)$ is a minimizer of $\s{B}$. 
	
\end{proof}

\subsection{Existence of solutions of \eqref{PB:dual-relaxed-plan}}\label{section_existence_minimizers_planning_problem} We will need the following preliminary result proved in \cite[Lemma 2.7]{CarGra}.
\begin{lemma}\label{one_holder_bound} Let $\alpha$ be a continuous function, and set
	$$
	\nu= \frac{r-d(q-1)}{d(q-1)(r-1)+rq},
	$$
	which by Hypothesis (H3) is positive.
	Then there exists $C>0$ such that for any smooth subsolution of $-\partial_t u + H(x,\nabla u) \leq \alpha$,
	$$
	u(t_1,x) \leq u(t_2,y) + C\left[|x-y|^{r'}(t_2-t_1)^{1-r'} +\left((t_2-t_1)^{\nu}\wedge 1 +T^{1/q}\right) \left(\|\alpha_{+}\|_{q'}+1\right)\right]
	$$
	for all $0\leq t_1 <t_2\leq T$ and $x$, $y\in \TT^{d}$. 
\end{lemma}
As a consequence of the previous lemma, we have that, for any  $x\in\TT^d$,
$$
u(0,x) \leq u( T,y) + C\left[|x-y|^{r'}T^{1-r'} +\left(T^{\nu}\wedge 1 +T^{1/q}\right) \left(\|\alpha_{+}\|_{q'}+1\right)\right] \hspace{0.3cm} \forall \; y\in \TT^{d}, 
$$
and since $x$ and $y$ belong to a bounded set, up to redefining $C$, we get
\begin{equation}
\label{inequality_with_infimum}
u( 0,x) \leq  \inf_{y\in \TT^{d}} u(T,y) +	C\left[T^{1-r'}+\left(T^{\nu}\wedge 1 +T^{1/q}\right) \left(\|\alpha_{+}\|_{q'}+1\right)\right]  \hspace{0.3cm} \forall \; x\in \TT^{d}.
\end{equation}

\begin{proposition}\label{prop:existence_solution_for_relaxed}
	Problem \eqref{PB:dual-relaxed-plan} admits at least one solution $(u,\alpha)$.  The function $u$ is H\"older continuous in $[0,T)\times \TT^d$,  $\alpha \geq 0$ a.e. and there exists  $C>0$, independent of $m_0$ and $m_T$, such that 
\begin{equation}
\label{uniform_bound_w_r_t_m_0_m_T}
\sup_{(t,x)\in [0,T] \times \TT^d} |u(t,x)| +\|\nabla u\|_{L^r}+\|\partial_t u \|_{\M}+\|\alpha\|_{L^{q'}}\leq C,
\end{equation}
where $\| \cdot \|_{\M}$ denotes the usual norm of $\M$ as dual space of $\mathcal C^0([0,T] \times \TT^d)$.
\end{proposition}
\begin{proof}
	Consider a smooth minimizing sequence $(u_n)_n$ for Problem \eqref{PB:dual2-planning}. 
	Using that $\A(u_n)= \A(u_n+c)$ for all $c\in \RR$, by subtracting $\min_{x \in \TT^d}u_{n}(T,x)$ we can suppose that  $\min_{x\in \TT^{d}} u_n(T,x)= 0$. For all $x\in\TT^d, t\in [0,T]$, let $$\alpha_n(t,x):=-\partial_t u_n(t,x) + H(x, \nabla u_n(t,x)).$$
	Then, inequality \eqref{inequality_with_infimum} applies to all $u_n$ giving 
	\begin{equation}
	\label{inequality_with_infimum_n}
	u_n(0,x) \leq  C\left[T^{1-r'}+\left(T^{\nu}\wedge 1 +T^{1/q}\right) \left(\|(\alpha_n)_{+}\|_{L^{q'}}+1\right)\right]  \hspace{0.3cm} \forall \; x\in \TT^{d}.
	\end{equation}
	Moreover,  Proposition \ref{prop:weakequalsstrong-plan} implies that $(u_n, \alpha_n)_n$ is a minimizing sequence for Problem \eqref{PB:dual-relaxed-plan}.  Hence, since Hypothesis (H2) and $F^*(x,a)=0 $ for all $a\leq0$, imply $\A(0)=0$, we have	\begin{align}
	0&\geq \int_0^T\int_{\TT^d} F^*(x,\alpha_n)\dd x \dd t+ \int_{\TT^d} u_n(T,x) \dd m_T(x)-\int_{\TT^d} u_n(0,x) \dd m_0(x) \label{alpha_n_bounded}\\
	&\geq \frac{1}{q'C} \|(\alpha_n)_+\|_{L^{q'}}^{q'} -C\left(T^{1-r'}+\left(T^{\nu}\wedge 1 +T^{1/q}\right) \left(\|(\alpha_n)_+\|_{L^{q'}}+1\right)\right), \nonumber
	\end{align}
	where we used the growth condition \eqref{eq:cost_growth_star} on $F^*$, inequality \eqref{inequality_with_infimum_n}, $u_n(T,\cdot)\geq 0$ and $\int_{\TT^d} m_0(x) \dd x=1$. We deduce that $((\alpha_n)_+)_n$ is a bounded sequence in $L^{q'}([0,T]\times\T^d)$ uniformly with respect to $m_0$ and $m_T$. 
	Moreover, there exists a constant $C_1>0$ such that $u_{n}(0,x) \leq C_1$ for all $x\in \TT^{d}$.
	
	In order to obtain uniform bounds on $(u_n)_n$, we need to modify the sequence. 
%
 Let $\eta \in C^{1}(\RR)$ such that $0\leq \eta' \leq 1$, $|\eta|\leq 2C_1$ and  $\eta(s)=s$ if $|s| \leq C_1$, and set $\tilde{u}_n := \eta \circ u_{n}$.  Since $u_{n}$ is Lipschitz continuous, we have that $\tilde{u}_n$ is Lipschitz continuous and, therefore, 
	\begin{align*}
	-\partial_t \tilde{u}_n + H(x,\nabla \tilde{u}_n) &\leq\eta'(\tilde {u}_n)\left( -\partial_t u_{n} +    H(x,\nabla u_{n})\right)\\
	&\leq  \eta'(\tilde{u}_n) (\alpha_n)_{+} \\
	 &\leq   (\alpha_n)_{+},
	\end{align*}
	where we have used that  $0\leq \eta' \leq 1$ and assumption (H2).
	
	Thus, $(\tilde{u}_n, (\alpha_n)_{+})_n \in \K$, $\|\tilde{u}_{n}\|_{L^\infty} \leq 2C_1$, i.e. $(\tilde u_n)_n$ is uniformly bounded,  and $(\tilde{u}_n, (\alpha_n)_{+})_n $ is a minimizing sequence. In fact, we have that $\tilde{u}_n(0,\cdot) \geq u_n(0,\cdot)$ and $\tilde{u}_n(T,\cdot)\leq u_{n}(T,\cdot)$ because $\eta(a)\geq a$ for all $a< 0$, $\eta(0)=0$ and $\eta(a)\leq a$ for all $a> 0$ (recall $u_{n}(0,x) \leq C_1$ and  $u_{n}(T,x) \geq 0$ for all $x\in \TT^{d}$). Moreover we can prove that $(\partial_t \tilde u_n)_n$ is bounded in $L^1([0,T]\times\T^d)$ and  $(\nabla \tilde u_n)$ is bounded in $L^r([0,T]\times\T^d;\R^d)$ uniformly w.r.t. $m_0$ and $m_T$. Indeed, by the growth condition \eqref{eq:hamiltonian_bounds} on $H$, for a.e. $(t,x)$, we have
	$$
	\partial_t \tilde{u}_n (t,x) + (\alpha_n)_{+}(t,x)+C\geq \frac{1}{Cr}|\nabla \tilde{u}_n|^{r}\geq 0.
	$$
	Therefore, since $ |\partial_t \tilde{u}_n | - | (\alpha_n)_{+}+C |\leq |\partial_t \tilde{u}_n + (\alpha_n)_{+}+C |$, we have 
	\begin{align*}
	\int_0^T\int_{\TT^d} |\partial_t \tilde{u}_n | \dd x \dd t &\leq  \int_0^T\int_{\TT^d}  (\alpha_n)_{+} \dd x\dd t+  \int_0^T\int_{\TT^d}\left(\partial_t \tilde{u}_n (t,x) + (\alpha_n)_{+}+C\right) \dd x \dd t \\ 
	&\leq C +  \int_{\TT^d} \left(\tilde u_n (T,x)-\tilde u_n (0,x)\right)\dd x \\
	&\leq C,
	\end{align*}
	where we used the fact that $((\alpha_n)_{+})_n$ is bounded in $L^{q'}([0,T]\times\T^d)$, hence in $L^1([0,T]\times\T^d)$, and that $(\tilde u_n )_n$ is uniformly bounded. Moreover,   
	$$
	\int_0^T\int_{\TT^d} |\nabla \tilde{u}_n(t,x)|^r \dd x \dd t \leq Cr \int_0^T\int_{\TT^d} \left(\partial_t \tilde{u}_n (t,x) + (\alpha_n)_{+}(t,x)+C\right) \dd x \dd t\leq C.
	$$
	Summarizing all the estimates, we have  
	\begin{equation}
	\label{stime}
	\|\tilde{u}_n\|_{L^\infty}+\|\nabla \tilde{u}_n\|_{L^r}+\|\partial_t \tilde{u}_n \|_{L^1}+\|(\alpha_n)_{+}\|_{L^{q'}}\leq C,
	\end{equation}
	with $C>0$ independent of $m_0$ and $m_T$.
	From this estimate we immediately deduce that, up to some subsequence,  $(\nabla \tilde{u}_n)_n$ weakly converges in $L^r([0,T]\times\T^d;\R^d)$, $(\partial_t \tilde{u}_n)_n$ weakly-* converges to a measure and $((\alpha_n)_{+})_n$ weakly converges in $L^{q'}([0,T]\times\T^d)$.
	
	Thanks to \cite[Lemma 1]{Car15} (see also \cite[Theorem 1.3]{CarSil}), we have that $(\tilde u_n)_n$ is a sequence of locally uniformly H\"older continuous functions on $[0,T)\times\TT^d$. Therefore, by the Arzel\`a-Ascoli theorem, we have that $(\tilde u_n)_n$ uniformly converges to $u \in \mathcal C^0([0,T)\times \TT^{d})$ on any compact set of $[0,T)\times\TT^d$. From \eqref{stime} we get that $u\in BV((0,T)\times \TT^{d})$ and $(\partial_t u, \nabla u)$ is the weak-* limit of $(\partial_t \tilde{u}_n, \nabla \tilde{u}_n )_n$.
	
	Let $\alpha\in L^{q'}([0,T]\times\T^d)$ be a weak limit of  $((\alpha_n)_{+})_n$ in $L^{q'}([0,T]\times\T^d)$.   Note that $\alpha \geq 0$ a.e. and, since $q'>1$, $\alpha$ is also a weak-* limit of  $((\alpha_n)_{+})_n$ in $L^1([0,T]\times\T^d)$. As a consequence of the last assertion, the pair $(u,\alpha)$ satisfies \eqref{uniform_bound_w_r_t_m_0_m_T} for some $C>0$. Now, take $\varphi$ a nonnegative test function in $\mathcal C^\infty_c([0,T]\times\TT^d)$, then for all $n$, we have
	$$
	\int_0^T\int_{\TT^d}-\partial_t\tilde u _n (t,x) \varphi (t,x)\dd x \dd t+ \int_0^T\int_{\TT^d} \varphi(t,x)H(x,\nabla \tilde{u}_n) \dd x \dd t\leq  \int_0^T\int_{\TT^d}\varphi(t,x) (\alpha_n)_{+} (t,x) \dd x \dd t.
	$$ 
	The first integral on the left hand side converge by the weak* convergence of $(\partial_t \tilde u_n)_n$ and the integral on the right hand side converge due to the weak  convergence of $((\alpha_n)_{+})_n$ in $L^{q'}([0,T]\times\T^d)$, while thanks to the convexity of $H$ in the gradient variable, we have
	$$
	\int_0^T\int_{\TT^d} \varphi(t,x) H(x,\nabla u)  \dd x \dd t\leq \liminf_{n\to\infty} \int_0^T\int_{\TT^d} \varphi(t,x) H(x,\nabla \tilde{u}_n) \dd x \dd t.
	$$
	Therefore,  $(u,\alpha)$ satisfies
	$$
	-\partial_t  u + H(x,\nabla u) \leq \alpha
	$$ 
	in the sense of distributions and in particular, $(u,\alpha)\in \K$.
	
	Let us now prove that $(u,\alpha)$ is a minimizer. Thanks to the convexity of $F^*$, we have the lower semicontinuity 
	\begin{equation}
	\label{convergence_F_ast_term}
	\int_0^T\int_{\TT^d} F^*(x,\alpha(t,x))  \dd x \dd t\leq \liminf_{n\to\infty} \int_0^T\int_{\TT^d}  F^*(x,(\alpha_n)_{+}(t,x)) \dd x \dd t.
	\end{equation}
	The uniform convergence of $(\tilde u_n)_n$ on any compact set of $[0,T)\times \TT^d$ implies that $(\tilde u_n(0,\cdot))_n$ converges uniformly to $u(0,\cdot)$, thus $\|u(0,\cdot)\|_{L^\infty} \leq C$ and 
\begin{equation}
\label{convergence_initial_value}
\int_{\TT^d} u(0,x) \dd m_0(x) = \lim _{n\to\infty}\int_{\TT^d}\tilde u_n(0,x) \dd m_0(x).
\end{equation}
The pointwise convergence of $(\tilde u_n(T,\cdot))_n$ is not ensured.  However, since $(\| \tilde u_n(T,\cdot)\|_{\infty})_n$ is uniformly bounded by $C$, there exists $g \in L^{\infty}(\TT^d)$ such that $\|g\|_{\infty} \leq C$ and, up to some subsequence, $u_n(T,\cdot)$ converges to $g$ in the weak-$^\ast$ topology $\sigma(L^{\infty}, L^{1})$. Therefore, 
\begin{equation}
\label{weak_convergence_l_infty}
\int_{\TT^d} \phi(x) g(x) \dd x = \lim_{n\to \infty} \int_{\TT^d} \phi(x) u_n(T,x) \dd x \hspace{0.4cm} \forall \; \phi \in L^{1}(\TT^d). 
\end{equation}
Now,  let $\phi \in \corrd{\mathcal C^0}(\TT^d)$. We have that
\begin{equation}
\label{computation_int_phi_u_n}
\begin{array}{rcl}
\ds\int_{\TT^d} \phi(x) u_n(T,x) \dd x&=&  \ds\int_{\TT^d} \int_{0}^{T} \phi(x) \partial_t  u_n(t,x)  \dd t \dd x +   \ds\int_{\TT^d} \phi(x) u_n(0,x) \dd x\\[5pt]
\; & \to &  \ds\int_{\TT^d} \int_{0}^{T} \phi(x) \partial_t  u(\dd t ,\dd x) + \int_{\TT^d} \phi(x) u(0,x) \dd x.
\end{array}
\end{equation}
Using that the trace $u(T,\cdot) \in L^{1}(\TT^d)$ of $u$ at $\{T\} \times \TT^d$ satisfies
$$
\int_{\TT^d} \int_{0}^{T} \phi(x) \partial_t  u(\dd t ,\dd x)= \int_{\TT^d}\phi(x) u(T,x) \dd x  - \int_{\TT^d} \phi(x) u(0,x) \dd x, 
$$
relations \eqref{weak_convergence_l_infty} and \eqref{computation_int_phi_u_n}  imply that $g=u(T,\cdot)$. Combining this result with \eqref{convergence_F_ast_term} and  \eqref{convergence_initial_value}, we deduce that $(u, \alpha)$ solves Problem \eqref{PB:dual-relaxed-plan}. The result follows. 
\end{proof}
\subsection{Stability result}
Now, for $\eps>0$ let us consider two probability densities $m^\eps_0$ and $m^\eps_T \in L^{1}(\TT^d)$ and denote by $(m_\eps, w_\eps)$ the unique solution to problem \eqref{Pb:mw2-planning} with $m_0$ and $m_T$ replaced by $m^\eps_0$ and $m^\eps_T \in L^{1}(\TT^d)$, respectively. Likewise, we denote by $(u_\eps, \alpha_\eps)\in \K$ a solution to the corresponding problem  \eqref{PB:dual-relaxed-plan} such that \eqref{uniform_bound_w_r_t_m_0_m_T} holds true.

The following stability result is a consequence of $\Gamma$-convergence and it follows easily from the statement and the proof of Proposition \ref{prop:existence_solution_for_relaxed}.
\begin{corollary} Suppose that, as $\eps \to 0^+$,   $(m^{\eps}_{0})_\eps$ and $(m^{\eps}_{T})_\eps $ converge  in $L^{1}(\TT^d)$ to $m_0$ and $m_T$, respectively. Then, the following assertions hold true: 
\begin{itemize}
\item[{\rm(i)}] $(m_\eps, w_\eps)$ converges weakly in $ L^q((0,T)\times \T^d)\times L^{\frac{r'q}{r'+q-1}}((0,T)\times \T^d;\R^d)$ to $(m, w)$, the unique solution to \eqref{Pb:mw2-planning}. \vspace{0.2cm}
\item[{\rm(ii)}] Up to some subsequence, $u_\eps\to u$ uniformly on every compact subset of $[0,T)\times \TT^d$, $u_\eps(T,\cdot) \to u(T,\cdot)$ weakly-$\star$ in $L^{\infty}(\TT^d)$,   and $\ds (\partial_t u_\eps, \nabla u_\eps,\alpha_\eps)\to (\partial_t u, \nabla u,\alpha)$ weakly-$\star$ in $\M((0,T)\times \TT^d)\times L^{r}((0,T)\times \TT^d;\R^d)\times L^{q'}((0,T)\times \TT^d)$, where $(u,\alpha)$ is a solution to \eqref{PB:dual-relaxed-plan} satisfying \eqref{PB:dual-relaxed-plan}.
\end{itemize}
\end{corollary}
\begin{proof} For $(u,\alpha) \in \K$ let us define
$$
\A_\eps(u,\alpha):=\int_0^T\int_{\T^d}  F^*(x,\alpha(t,x)) \dd x\dd t
+ \int_{\T^d} u(T,x)m_T^\eps(x) \dd x - \int_{\T^d} u(0,x)m_0^\eps(x) \dd x.
$$
Define also $\K_1^\eps$ as $\K_1$ with $m_0$ and $m_T$ replaced by  $m_0^\eps$ and $m_T^\eps$, respectively. 

Notice that Proposition \ref{prop:existence_solution_for_relaxed} implies that for all $\eps>0$ we have
\begin{equation}
\label{redefinition_function_A}
\inf_{(u,\alpha)\in \K} \A_\eps(u,\alpha)= \inf \left\{ \A_\eps(u,\alpha) \; | \; (u,\alpha) \in \K, \; \; \|u(0,\cdot)\|_{L^\infty} \leq C, \; \mbox{and } \|u(T,\cdot)\|_{L^\infty} \leq C  \right\}.
\end{equation}
Using that 
$$
|\A_\eps(u,\alpha)-\A(u,\alpha)| \leq C\left( \| m_0^\eps - m_0\|_{L^1} + \| m_T^\eps - m_T\|_{L^1}\right),
$$
for all $(u,\alpha) \in \K$ such that $\|u(0,\cdot)\|_{L^\infty} \leq C$ and $\|u(T,\cdot)\|_{L^\infty} \leq C$, relation \eqref{redefinition_function_A} implies that 
$$
\lim_{\eps \to 0^+}\ds -\min_{(m,w) \in \K_1^\eps} \B(m,w)=  \lim_{\eps \to 0^+} \inf_{(u,\alpha) \in \K} \A_\eps(u,\alpha)=  \inf_{(u,\alpha) \in \K} \A(u,\alpha)=-\min_{(m,w) \in \K_1} \B(m,w).
$$
Arguing as in the proof of Proposition \ref{prop:existence_solution_for_relaxed}, we have the existence of $(u,\alpha) \in \K$ such that, up to some subsequence, $(u_\eps,\alpha_\eps)_\e$ converges to $(u,\alpha)$ in the sense of {\rm(ii)}, and 
$$
\A(u,\alpha) \leq \lim_{\eps \to 0^+} \A_{\eps}(u,\alpha)= \inf_{(u,\alpha) \in \K} \A(u,\alpha),
$$ 
which implies {\rm(ii)}. In addition, Lemma \ref{Lem:dualite-planning} yields that $(m_\eps,w_\eps)$ is uniformly bounded in $L^{q}((0,T)\times \TT^d) \times L^{\ell}((0,T)\times \TT^d; \RR^d)$, where $\ell:=\frac{r'q}{r'+q-1}$. Then, the lower semicontinuity of the convex functional $\B$  implies that any weak limit point $(m,w)$ of $((m_\eps,w_\eps))_\eps$ satisfies 
$$\B(m,w) \leq \lim_{\eps \to 0} \B(m_\eps, w_\eps)= \min_{(m',w') \in \K_1} \B(m',w').$$
Since $(m_\eps, w_\eps)$ satisfies \eqref{continuity_equation_prescribed_initial_and_final_conditions} with initial and final conditions given by $m_0^\eps$ and $m_T^\eps$, respectively, we can pass to the limit in that equation to obtain that $(m,w)$ also satisfies \eqref{continuity_equation_prescribed_initial_and_final_conditions} with initial and final conditions given by $m_0$ and $m_T$, respectively. Finally, since $m_\eps \geq 0$ a.e. we also get that $m \geq 0$ a.e., which implies that $(m,w) \in \K_1$. Therefore, $(m,w)$ is the unique solution to \eqref{Pb:mw2-planning} and the whole sequence $(m_\eps,w_\eps)_\e$ converges to $(m,w)$ weakly in $L^q((0,T)\times \T^d)\times L^{\ell}((0,T)\times \T^d;\R^d)$. The result follows. 
\end{proof}

\subsection{Uniqueness} \label{sec:uniqueness planning}

In this subsection address uniqueness of solutions to the planning problem.
Let $(\bar u,\bar m)$ be a weak solution to \eqref{eq:planning_general}. In light of Theorem \ref{theo:solution-equals-optimizer-planning},  the pair $(\bar m, \bar w) = (\bar m,-\bar mD_\xi H(\cdot,\nabla \bar u))$ is the minimizer of \eqref{Pb:mw2-planning} while $(\bar u, f(\cdot,\bar m))$ is a solution of \eqref{PB:dual-relaxed-plan}. In particular,  $\bar m$ and $\bar w$ are unique because of the uniqueness of the solution of \eqref{Pb:mw2-planning}.

On the other hand, if $H$ is strictly convex in the second variable, then uniqueness of $\bar w$ implies that $\nabla \bar u$ is unique on the set $\{\bar m > 0\}$ (Cf.~the statement of Theorem 6.15 in \cite{OrrPorSav}).

\section{Sobolev regularity of weak solutions} \label{sec:regularity mfg}

In this section, by applying the techniques used in \cite{GraMes}, we prove some additional \emph{a priori} regularity for the weak solutions for  the planning problem \eqref{eq:planning_general}.
(The definition of weak solution is given  in Sections \ref{sec:weak solutions planning}.)
We need to assume the following hypotheses.

\vspace{10pt}

\noindent{\bf Additional assumptions}

\vspace{10pt}

\begin{enumerate}
	\item[(H\arabic{hyp})] (Conditions on the coupling)
	The function $f$ satisfies
	\begin{equation}
	\label{f Lipschitz in x}
	|f(x,m) - f(y,m)| \leq C(m^{q-1}+1)|x-y|\  \ \forall x,y \in \bb{T}^d, \ m \geq 0.
	\end{equation}
Moreover, there exists $c_f > 0$ such that
	\begin{equation} \label{f strongly monotone}
	\left( f(x,\tilde m) - f(x,m)\right)(\tilde m - m) \geq c_f\min\{\tilde m^{q-2},m^{q-2}\}|\tilde m-m|^2 \ \forall \tilde m, m \geq 0, \ \tilde m \neq m.
	\end{equation}
	If $q < 2$ one should interpret $0^{q-2}$ as $+\infty$ in \eqref{f strongly monotone}.
	In this way, when $\tilde m = 0$, for instance, \eqref{f strongly monotone} reduces to $f(x,m)m \geq c_f m^q$, as in the more regular case $q \geq 2$.
	
	\stepcounter{hyp}
	\item[(H\arabic{hyp})] (Coercivity assumptions.) There exist $j_1,j_2:\bb{R}^d \to \bb{R}^d$ and $c_H > 0$ such that
	\begin{equation} \label{eq:Hcoercivity}
	H(x,\xi) + H^*(x,\zeta) - \xi \cdot \zeta \geq c_H|j_1(\xi) - j_2(\zeta)|^2.
	\end{equation}
	In particular, and in light of our restriction \eqref{eq:hamiltonian_bounds}, we will have that $j_1(\xi) \sim |\xi|^{r/2-1}\xi$ and $j_2(\zeta) \sim |\zeta|^{r'/2-1}\zeta$.
	\stepcounter{hyp}
\end{enumerate}

\subsection{Global space regularity} \label{sec:space regularity}

By using arguments analogous to those in \cite[Proposition 4.3]{GraMes}, we get
\begin{proposition} \label{prop:space-regularity}
 Assume $(H1)$, $(H3)$, $(H5)$ and $(H6)$,  $m_0, m_T \in W^{2,1}(\bb{T}^d)$ and that $H^*$ is twice continuously differentiable in $x$ with	
	\begin{equation} \label{hyp:D_x^2 H}
	|D_{xx}^2 H^*(x,\zeta)| \leq C|\zeta|^{r'} + C. \tag{H\arabic{hyp}}
	\end{equation}
	\stepcounter{hyp}
Then, if $(u,m)$  is a weak solution of the planning problem \eqref{eq:planning_general}, we have 
$$\|m^{\frac{q}{2} - 1}\nabla m\|_{L^2([0,T]\times\T^d)} \leq C \hspace{0.3cm} \mbox{and } \; \; \|m^{1/2}D (j_1(\nabla u))\|_{L^2([0,T]\times\T^d)} \leq C.$$
\end{proposition}

\begin{proof}
	We give only a sketch, leaving the reader to find the remaining details in \cite[Proposition 4.3]{GraMes}.
	From the proof of Proposition \ref{prop:existence_solution_for_relaxed}, there exists a sequence $(u_n,\alpha_n)$ such that $u_n \in C^1$, $\alpha_n$ is continuous, and
	$$
	-\partial_t u_n + H(x,\nabla u_n) \leq \alpha_n;
	$$ 
	moreover, $\alpha_n \rightharpoonup f(\cdot,m)$ weakly in $L^{q'}$, $\tilde u_n \to u$ locally uniformly in $[0,T) \times \bb{T}^d$, $\nabla \tilde u_n \rightharpoonup \nabla u$ weakly in $L^r$, and $\partial_t \tilde u_n \rightharpoonup \partial_t u$ weakly-$\star$ in the space of Radon measures.

	Fix $\delta \in \bb{T}^d$.
	For any function $f$ on $[0,T] \times \bb{T}^d$, $f^\delta(t,x) := f(t,x+\delta)$.
	Use $u^\delta_n$ as a test function in $\partial_t m + \nabla \cdot w = 0$ to get
	\begin{equation} \label{eq:phi_n^delta-test-m}
	\int_{\bb{T}^d} u_n^\delta(T) m_T - u_n^\delta(0)m_0 \geq \int_0^T \int_{\bb{T}^d} (H(x+\delta,\nabla u_n^\delta)-\alpha_n^\delta)m + \nabla u_n^\delta \cdot w \dd x\dd t.
	\end{equation}
	Combine this with the optimality condition \eqref{defcondsup-plan} to get
	\begin{multline} \label{eq:space-regularity1'}
	\int_{\bb{T}^d} (u_n^\delta(T) - u(T)) m_T - (u^\delta_n(0)-u(0))m_0 
	\\
	\geq
	\int_0^T \int_{\bb{T}^d} (H(x+\delta,\nabla u^\delta_n) + H^*(x,-w/m) + \nabla u^\delta_n \cdot w/m - \alpha_n^\delta + f(m))m  \dd x\dd t.
	\end{multline}
	Similarly, using $u_n$ as a test function in $\partial_t m^\delta + \nabla \cdot w^\delta = 0$ and combining with \eqref{defcondsup-plan},
	\begin{multline} \label{eq:space-regularity2'}
	\int_{\bb{T}^d} (u_n(T) - u^\delta(T)) m^\delta(T) - (u_n(0)-u_n^\delta(0))m^\delta_0 
	\\
	\geq
	\int_0^T \int_{\bb{T}^d} (H(x,\nabla u_n) + H^*(x+\delta,-w^\delta/m^\delta) + \nabla u_n \cdot w^\delta/m^\delta - \alpha_n + f^\delta(m^\delta))m^\delta  \dd x\dd t
	\end{multline}
	Combining \eqref{eq:space-regularity1'} and \eqref{eq:space-regularity2'}, after some changes of variables (translations) and a Taylor expansion of $H^*$, we deduce
	\begin{multline} \label{eq:space-regularity1}
	\int_0^T \int_{\bb{T}^d} (H(x+\d,\nabla u^\delta_n) + H^*(x+\d,-w/m) + \nabla u^\delta_n \cdot w/m)m  \dd x\dd t
	\\
	+ \int_0^T \int_{\bb{T}^d} (H(x-\d,\nabla u^{-\delta}_n) + H^*(x-\d,-w/m) + \nabla u_n^{-\delta} \cdot w/m)m  \dd x\dd t
	\\
	\leq \int_{\bb{T}^d} (u_n(T)(m_T^\delta + m_T^{-\delta}) - 2u(T)m_T)\dd x
	-  \int_{\bb{T}^d} (u_n(0)(m_0^\delta + m_0^{-\delta}) - 2u(0)m_0)\dd x
	\\
	+ \int_0^T \int_{\bb{T}^d} \left(\alpha_n^\delta + \alpha_n^{-\delta} - 2f(m)\right)m \dd x \dd t
	+ \int_0^T \int_{\bb{T}^d} \int_0^1\int_s^{-s}\langle D^2_{xx}H^*(x+\tau\d,-w/m)\d,\d\rangle m\dd\tau\dd s\dd x\dd t.
	\end{multline} 
	Equation \eqref{eq:space-regularity1} can be obtained by using $u_n^\delta$ as a test function in \eqref{eq:planning_general}(ii) and $u_n$ as a test function in the same equation with $m$ replaced by $m^\delta$, then using the optimality condition \eqref{defcondsup-plan}.

	Letting $n \to \infty$ in \eqref{eq:space-regularity1}, we get
	\begin{multline} \label{eq:space-regularity2}
	\int_0^T \int_{\bb{T}^d} (H(x+\d,\nabla u^\delta) + H^*(x+\d,-w/m) + \nabla u^\delta \cdot w/m)m  \dd x\dd t
	\\
	+ \int_0^T \int_{\bb{T}^d} (H(x-\d,\nabla u^{-\delta}) + H^*(x-\d,-w/m) + \nabla u^{-\delta} \cdot w/m)m  \dd x\dd t
	\\
	\leq \int_{\bb{T}^d}  u(T)(m_T^\delta + m_T^{-\delta} - 2 m_T)\dd x
		-  \int_{\bb{T}^d} u(0)(m_0^\delta + m_0^{-\delta} - 2m_0)\dd x
	\\
	+ \int_0^T \int_{\bb{T}^d} \left(f^\delta(m^\delta) + f^{-\delta}(m^{-\delta}) - 2f(m)\right)m \dd x \dd t
	+  C |\d|^2\int_0^T\int_{\T^d}(|w/m|^{r'}+1)m\dd x\dd t,
	\end{multline} 
	where we have used Hypothesis \eqref{hyp:D_x^2 H}.
	We have (see \cite[computation (4.25)]{GraMes})
	\begin{multline} \label{eq:g monotonicity estimates}
	\int_{\bb{T}^d} \left(f^\delta(m^\delta) + f^{-\delta}(m^{-\delta}) - 2f(m)\right){m} \dd x
	\\
	\leq
	C|\delta|^2 \left(1+\int_{\bb{T}^d} \min\{ m^\delta,m\}^{q} \dd x\right)
	- \frac{c_f}{2}\int_{\bb{T}^d} \min\{ (m^\delta)^{q-2}, m^{q-2}\}| m^\delta - m|^2 \dd x.
	\end{multline}
	 Combining this estimate with assumption \eqref{eq:Hcoercivity}, and using the inequality $|a+b|^2 \leq 2\left( a^2 +b^2\right)$ for all $a$, $b\in \RR$, we get
	\begin{multline} 
	\frac{c_H}{2}\int_0^T \int_{\bb{T}^d} \left(|j_1(\nabla u^\delta) - j_1(\nabla u^{-\delta})|^2\right) m \dd x \dd t
	+ \frac{c_f}{2}\int_0^T \int_{\bb{T}^d} \min\{(m^\delta)^{q-2},m^{q-2}\}|m^\delta - m|^2 \dd x \dd t
	\\
	\leq    C |\d|^2\int_0^T\int_{\T^d}(|w/m|^{r'}+1)m\dd x\dd t +C |\d|^2\left(T+\int_0^T\int_{\bb{T}^d} \min\{m^\delta,m\}^{q} \dd x \dd t\right) \\
	+ C|\delta|^2 \|m_T\|_{W^{2,1}}\|u(T)\|_\infty+ C|\delta|^2 \|m_0\|_{W^{2,1}}\|u(0)\|_\infty,
	\end{multline}
	 where we have used the $L^\infty$ estimate on $u(0), u(T)$ from Proposition \ref{prop:existence_solution_for_relaxed}.
	Since also
	$$
	 \int_0^T\int_{\T^d}(|w/m|^{r'}+1)m\dd x\dd t +\lim_{\delta \to 0}\int_0^T\int_{\bb{T}^d} \min\{ m^\delta, m \}^{q} \dd x\dd t \leq C\s{B}(m,w) + C < \infty
	$$
	we conclude that there exists some $C$ such that
	\begin{multline} 
	\frac{c_H}{2}\int_0^T \int_{\bb{T}^d} \left(|j_1(\nabla u^\delta) - j_1(\nabla u^{-\delta})|^2\right) m \dd x \dd t
	+ \frac{c_f}{2}\int_0^T \int_{\bb{T}^d} \min\{(m^\delta)^{q-2},m^{q-2}\}|m^\delta - m|^2 \dd x \dd t
	\leq C |\d|^2.
	\end{multline}
	Dividing by $|\d|^2$ and letting $\d\to 0$, we easily obtain the result. 
\end{proof}

\subsection{Local time regularity} \label{sec:time regularity} We rely on very similar arguments as those found in the  previous section, but applied to time rather than space.
Our translations in time will be localized so as to avoid conflict with the initial-final conditions.

Let $\e\in\R$ be small and $\eta:[0,T]\to[0,1]$ be smooth and compactly supported on $(0,T)$ such that $|\e|<\min\left\{{\rm{dist}}(0,{\rm{spt}}(\eta));{\rm{dist}}(T,{\rm{spt}}(\eta))\right\}.$ For competitors $(u,\a)$ of the minimization problem for $\cA$, let us define
$$u^\e(t,x):=u(t+\e\eta(t),x);\ \ \a^\e(t,x):=(1+\e\eta'(t))\a(t+\e\eta(t),x).$$
Notice that by construction, if $t\in\{0,T\}$ then $u(t,x)=u^\e(t,x)$ and $\a(t,x)=\a^\e(t,x)$, provided that $\a(t,x)$ is well-defined. 

Similarly, for competitors $(m,w)$  of minimization problem for $\cB$, we define
$$m^\e(t,x):=m(t+\e\eta(t),x);\ \ w^\e(t,x):=(1+\e\eta'(t))w(t+\e\eta(t),x)$$
and here as well if $t\in\{0,T\}$ then $m(t,x)=m^\e(t,x)$ and $w(t,x)=w^\e(t,x).$

We define moreover perturbations on the data as 
$$f^\e(t,x,m):=(1+\e\eta'(t)) f(x,m);\ \ F^\e(t,x,m):= (1+\e\eta'(t)) F(x,m),$$
from which the Legendre transform w.r.t.~the last variable satisfies
$$(F^\e)^*(t,x,\a):= (1+\e\eta'(t)) F^*(x,\a/(1+\e\eta'(t))).$$
Finally, we define
$$H^\e(t,x,\xi):= (1+\e\eta'(t)) H(x,\xi), \ \ {\rm{thus}}\ (H^\e)^*(t,x,\zeta):= (1+\e\eta'(t)) H^*(x,\zeta/(1+\e\eta'(t))).$$
We define the functional $\cA^\e$ and its domain $\s{K}^\e$ the same way as $\cA$ and $\s{K}$ in Section \ref{sec:2opti-planning}, with the data $ H^\e, (F^\e)^*$ replacing $H, F^*$.
Likewise, we define the functional $\cB^\e$ and its domain $\s{K}_1^\e$ the same way as  $\cB$ and $\s{K}_1$ in Section \ref{sec:2opti-planning}
 with data $(H^\e)^*, F^\e$ replacing $H^*, F$.
A very important remark is that the following (cf.~\cite[Section 4.1]{GraMes}.):
\begin{lemma}
	$(u,\a) \in \s{K}$ is a minimizer of the problem for $\cA$ if and only if $(u^\e,\a^\e)$ is a minimizer of the problem for $\cA^\e$.
	Similarly, $(m,w)$ is a minimizer of the problem for $\cB$ if and only if $(m^\e,w^\e)$ is a minimizer of the problem for $\cB^\e$.
\end{lemma}
\begin{proof}
	After a change of variables $t \mapsto t + \e \eta(t)$, we observe that $(u,\a) \in \s{K}$ if and only if $(u^\e,\a^\e) \in \s{K}^\e$, and moreover $\s{A}^\e(u^\e,\a^\e) = \s{A}(u,\a)$ for all $(u,\a) \in \s{K}$.
	The first claim follows.
	The proof of the second claim is analogous.
\end{proof}

In the same spirit as Proposition \ref{prop:space-regularity}, we can formulate 
\begin{proposition}\label{prop:time-regularity}
Assume $(H1)$, $(H3)$, $(H4)$, relation \eqref{f strongly monotone} in $(H5)$, $(H6)$ and
\begin{equation}\label{hyp:H*}
|D_\zeta H^*(x,\zeta)\cdot\zeta|\leq C|\zeta|^{r'}+C, \ \ |D^2_{\zeta\zeta} H^*(x,\zeta)|\leq C|\zeta|^{r'-2}\ \ {\rm{a.e.}}\ \zeta\in\R^d, \tag{H\arabic{hyp}}
\end{equation} \vspace{-0.5cm}
\stepcounter{hyp}
\begin{equation}\label{hyp:j_2}
|D_\zeta j_2(\zeta)\cdot\zeta|^2\le C|\zeta|^{r'}+C, \ \ {\rm{a.e.}}\ \zeta\in\R^d. \tag{H\arabic{hyp}}
\end{equation}
\stepcounter{hyp}
Then, if $(u,m)$  is a weak solution of the planning problem \eqref{eq:planning_general}, we have
$${\color{red}|}m^{1/2}\partial_t(j_1(\nabla u)){\color{red}|}\in L^2_{\rm{loc}}((0,T);L^2(\T^d)), \hspace{0.4cm} \partial_t(m^{q/2})\in L^2_{\rm{loc}}((0,T);L^2(\T^d)).$$
Moreover, the norms of $\big|m^{1/2}\partial_t(j_1(\nabla u))\big|$ and $\partial_t(m^{q/2})$ in $L^2_{\rm{loc}}((0,T);L^2(\T^d))$ depend only on the data.
\end{proposition}

\begin{uncomment}
\begin{remark}\label{rem:no_time_diff}
We cannot conclude about the same results as in the previous proposition in the case of nonzero diffusion matrix $A$. It seems that there is a technical reason behind this, which cannot be handled by our approach. For more details, we refer to the terms $I_i^{\e,n}$ in the proof below, on which we cannot obtain suitable estimates.
\end{remark}
\end{uncomment}

\begin{proof}[Proof of Proposition \ref{prop:time-regularity}]

The proof closely follows the steps of Proposition \ref{prop:space-regularity}, but with translations in time rather than space.

{\it Step 0. Preparatory step.}
Take the sequence $(u_n,\alpha_n)_{n\ge 0}$ defined as in the proof of Proposition \ref{prop:space-regularity}.
Now use $u_n$ as test function for $\partial_t m^\e+\nabla\cdot w^\e=0$. In the same way, use $u_n^\e$ (defined as $u^\e$) as test function for $\partial_t m +\nabla\cdot w=0$.  
By the same derivation as for \eqref{eq:space-regularity1'} and \eqref{eq:space-regularity2'} one obtains
\begin{align}\label{eq:comb1}
&\int_{\T^d}[u_n(T)-u(T)]m_T-[u_n(0)-u(0)]m_0\dd x\\
\nonumber&\geq
\int_0^T\int_{\T^d}\left[H^\e(t,x,\nabla u_n^\e)+H^*(x,-w/m)+\nabla u_n^\e\cdot w/m + f(x,m)-\a_n^\e\right]m\dd x\dd t\\
\nonumber&=
\int_0^T\int_{\T^d}\left[H^\e(t,x,\nabla u_n^\e)+(H^\e)^*(t,x,-w/m)+\nabla u_n^\e\cdot w/m + f(x,m)-\a_n^\e\right]m\dd x\dd t\\
\nonumber&+\int_0^T\int_{\T^d}\left[H^*(x,-w/m)-(H^\e)^*(t,x,-w/m)\right]m\dd x\dd t
\end{align}
and
\begin{align}\label{eq:comb2}
&\int_{\T^d}[u_n(T)-u(T)]m_T-[u_n(0)-u(0)]m_0\dd x\\
\nonumber&\geq
\int_0^T\int_{\T^d}\left[H(x,\nabla u_n)+(H^\e)^*(t,x,-w^\e/m^\e)+\nabla u_n\cdot w^\e/m^\e + f^\e(t,x,m^\e)-\a_n\right]m^\e\dd x\dd t\\
\nonumber&=
\int_0^T\int_{\T^d}\left[H^{-\e}(s,x,\nabla u_n^{-\e})+H^*(x,-w/m)+\nabla u_n^{-\e}\cdot w/m\right]m\dd x\dd s\\
\nonumber&\quad +\int_0^T\int_{\T^d}O(\e^2)H(x,\nabla u_n^{-\e})m\dd x\dd s
 + \int_0^T\int_{\T^d}\left[f^\e(t,x,m^\e)-\a_n\right]m^\e\dd x\dd t\\
\nonumber&= 
\int_0^T\int_{\T^d}\left[H^{-\e}(s,x,\nabla u_n^{-\e})+(H^{-\e})^*(s,x,-w/m)+\nabla u_n^{-\e}\cdot w/m\right]m\dd x\dd s
\\
\nonumber&\quad +\int_0^T\int_{\T^d}\left[H^*(x,-w/m)-(H^{-\e})^*(s,x,-w/m)\right]m\dd x\dd s+\int_0^T\int_{\T^d}O(\e^2)H(x,\nabla u_n^{-\e})m\dd x\dd s
\\
\nonumber&\quad + \int_0^T\int_{\T^d}\left[f^\e(t,x,m^\e)-\a_n\right]m^\e\dd x\dd t
\end{align}
where in the penultimate equation we used the change of variable $s=t+\e\eta(t)$ (which means in particular that $t=s-\e\eta(s)+O(\e^2)$ and $\frac{1}{1+\e\eta'(t)}=1-\e\eta'(s)+O(\e^2)$). By slight abuse of notation we denoted 
$$u_n^{-\e}(s,x):=u_n(s-\e\eta(s)+O(\e^2),x), 
$$
and we use the original notation for $H^{-\e}$ and $(H^{-\e})^*$.
Adding \eqref{eq:comb1} to \eqref{eq:comb2} we get
\begin{multline} \label{eq:comb3}
\int_0^T\int_{\T^d}\left[H^{-\e}(t,x,\nabla u_n^{-\e})+(H^{-\e})^*(t,x,-w/m)+\nabla u_n^{-\e}\cdot w/m\right]m\dd x\dd t
\\
+ \int_0^T\int_{\T^d}\left[H^\e(t,x,\nabla u_n^\e)+(H^\e)^*(t,x,-w/m)+\nabla u_n^\e\cdot w/m \right]m\dd x\dd t
\\
\leq 
\int_0^T\int_{\T^d}\left[\a_n^\e - f(x,m)\right]m\dd x\dd t
+ \int_0^T\int_{\T^d}\left[\a_n - f^\e(t,x,m^\e)\right]m^\e\dd x\dd t\\
+ 2\int_{\T^d}[u_n(T)-u(T)]m_T-[u_n(0)-u(0)]m_0\dd x
+ R_n(\e),
\end{multline}
where the remainder term satisfies
\begin{multline*}
R_n(\e) = \int_0^T\int_{\T^d}\left[(H^\e)^*(t,x,-w/m)+(H^{-\e})^*(t,x,-w/m)-2H^*(x,-w/m)\right]m\dd x\dd t
\\
+ O(\e^2)\int_0^T\int_{\T^d}H(x,\nabla u_n^{-\e})m\dd x\dd t.
\end{multline*}

{\it Step 1. Error term.}
Before letting $n \to \infty$ let us first show that $R_n(\e) = O(\e^2)$ (uniformly in $n$).
To that end we estimate the terms $H^* - (H^\e)^*$ and $H^* - (H^{-\e})^*$.
By a Taylor expansion, we have
\begin{align*}
(H^\e)^*(t,x,\zeta)&=(1+\e\eta'(t))H^*(x,\zeta/(1+\e\eta'(t)))=(1+\e\eta'(t))H^*(x,(1-\e\eta'(t)+O(\e^2))\zeta)\\
&=(1+\e\eta'(t))\left[H^*(x,\zeta)-\e\eta'(t)D_\zeta H^*(x,\zeta)\cdot\zeta+O(\e^2)D_\zeta H^*(x,\zeta)\cdot\zeta\right]\\
& \quad +(1+\e\eta'(t))\left[\e\eta'(t)+O(\e^2)\right]^2 \frac{1}{2}D^2_{\zeta\zeta}H^*(x,\zeta^*_\e)\zeta\cdot\zeta
\end{align*}
where $\zeta^*_\e$ is a point on the segment between $\zeta$ and $(1-\e\eta'(t)+O(\e^2))\zeta$. Let us notice that due to the growth condition \eqref{hyp:H*}
we have that 
$$|D^2_{\zeta\zeta} H^*(x,\zeta^*_\e)\zeta\cdot\zeta|\le C|\zeta^*_\e|^{r'-2}|\zeta|^2,$$
where, by the comparison $(1-C|\e|)|\zeta|\le|\zeta^*_\e|\le(1+C|\e|)|\zeta|$, the right-hand side is finite even when $\zeta_{\e}^* = 0$, and in particular
$$|D^2_{\zeta\zeta} H^*(x,\zeta^*_\e)\zeta\cdot\zeta|\leq C|\zeta|^{r'}$$ 
for small enough $\e$.
Therefore, by \eqref{hyp:H*} we have
\begin{multline} \label{eq:He*-H*}
(H^\e)^*(t,x,\zeta) - (1+\e\eta'(t))H^*(x,\zeta)+\e\eta'(t)D_\zeta H^*(x,\zeta)\cdot\zeta \\
= O(\e^2)\left[D_\zeta H^*(x,\zeta)\cdot\zeta
+  \frac{1}{2}D^2_{\zeta\zeta}H^*(x,\zeta^*_\e)\zeta\cdot\zeta\right]
= O(\e^2)(|\zeta|^{r'} + 1).
\end{multline}
Using a similar argument, we deduce
\begin{equation} \label{eq:H-e*-H*}
(H^{-\e})^*(t,x,\zeta) - (1-\e\eta'(t))H^*(x,\zeta)-\e\eta'(t)D_\zeta H^*(x,\zeta)\cdot\zeta 
= O(\e^2)(|\zeta|^{r'} + 1).
\end{equation}
Adding together \eqref{eq:He*-H*} and \eqref{eq:H-e*-H*}, setting $\zeta = -w/m$ and then integrating against $m$, we get
\begin{multline}
\int_0^T\int_{\T^d}\left[(H^\e)^*(t,x,-w/m)+(H^{-\e})^*(t,x,-w/m)-2H^*(x,-w/m)\right]m\dd x\dd t\\
=\int_0^T\int_{\T^d} O(\e^2)\left|\frac{w}{m}\right|^{r'}m\dd x\dd t 
=O(\e^2),
\end{multline}
where in the last equation we used the assumption \eqref{eq:hamiltonian_conjugate_bounds} and the fact that $\cB(m,w)$ is finite.

As for what remains of $R_n(\e)$, we use $u_n^{-\e}$ as a test function in $\partial_t m +\nabla\cdot w=0$; with the appropriate change of variable we get
\begin{multline}\label{error_term}
\int_0^T\int_{\T^d}H(x,\nabla u_n^{-\e})m\dd x\dd t
\\
\leq (1+O(\e))\left[\int_{\T^d}u_n(T) m_T-u_n(0)m_0\dd x + \int_0^T\int_{\T^d}(\alpha_n)_+(s-\eta(s)+O(\e^2),x)m(s,x)\dd x\dd s\right].
\end{multline}
Recall that from the definition of weak solution that $m \in L^q((0,T)\times\T^d)$, that $m_0, m_T \in L^1(\T^d)$ by hypothesis, while from the proof of Proposition \ref{prop:existence_solution_for_relaxed} we have that $(\alpha_n)_+$ is bounded in $L^{q'}$,  and $u_n$ can be taken such that $u_{n}(0)$ and $u_n(T)$ are bounded in  $L^{\infty}$. It follows that $\int_0^T\int_{\T^d}H(x,\nabla u_n^{-\e})m\dd x\dd s$ is bounded.
We can now rewrite \eqref{eq:comb3} as
\begin{multline} \label{eq:comb4}
\int_0^T\int_{\T^d}\left[H^{-\e}(t,x,\nabla u_n^{-\e})+(H^{-\e})^*(t,x,-w/m)+\nabla u_n^{-\e}\cdot w/m\right]m\dd x\dd t
\\
+ \int_0^T\int_{\T^d}\left[H^\e(t,x,\nabla u_n^\e)+(H^\e)^*(x,-w/m)+\nabla u_n^\e\cdot w/m \right]m\dd x\dd t
\\
\leq
\int_0^T\int_{\T^d}\left[\a_n^\e - f(x,m)\right]m\dd x\dd t
+ \int_0^T\int_{\T^d}\left[\a_n - f^\e(t,x,m^\e)\right]m^\e\dd x\dd t\\
+ 2\int_{\T^d}[u_n(T)-u(T)]m_T-[u_n(0)-u(0)]m_0\dd x
+ O(\e^2).
\end{multline}

{\it Step 2. Taking $n \to \infty$.}
We can now proceed exactly as in the proof of \cite[Proposition 4.3.]{GraMes} when taking limits as $n\to+\infty$ in \eqref{eq:comb4}.
First notice that we have the weak convergence (up to a subsequence) of $\nabla u_n^{-\e} \rightharpoonup \nabla u^{-\e}, \nabla u_n^{\e} \rightharpoonup \nabla u^{\e}$ in $L^r((0,T)\times\T^d)$.
Second recall that $\alpha_n$ converges weakly in $L^{q'}((0,T)\times\T^d)$ to $f$, and likewise $\a_n^\e$ converges weakly to $f^\e(\cdot,\cdot,m^\e)$.
Third recall that $u_n(T)$ converges to $u(T)$ in the weak-$\star$ topology $\sigma(L^\infty,L^1)$, and $u_n(0) \to u(0)$ uniformly (see the proof of Proposition \ref{prop:existence_solution_for_relaxed}).
Arguing as in \cite{GraMes}, we deduce that
\begin{multline} \label{eq:comb5}
\int_0^T\int_{\T^d}\left[H^{-\e}(t,x,\nabla u^{-\e})+(H^{-\e})^*(t,x,-w/m)+\nabla u^{-\e}\cdot w/m\right]m\dd x\dd t
\\
+ \int_0^T\int_{\T^d}\left[H^\e(t,x,\nabla u^\e)+(H^\e)^*(t,x,-w/m)+\nabla u^\e\cdot w/m \right]m\dd x\dd t
\\
\leq 
-\int_0^T\int_{\T^d}\left(f^\e(t,x,m^\e) - f(x,m)\right)(m^\e - m)\dd x\dd t
+ O(\e^2).
\end{multline}\\

{\it Step 3. Time regularity for $u$.} 

By the coercivity condition on $H$ and $H^*$ for any $\g>0$ we have that
$$\g H(x,\xi)+\g H^*(x,\zeta)-\g\xi\cdot\zeta\ge\g c_H |j_1(\xi)-j_2(\zeta)|^2,\ \ \forall x\in\T^d,\xi,\zeta\in\R^d.$$
In particular, setting $\tilde\zeta:=\g\zeta$, this implies
$$\g H(x,\xi)+\g H^*(x,\tilde\zeta/\g)-\xi\cdot\tilde\zeta\ge\g c_H|j_1(\xi)-j_2(\tilde\zeta/\g)|^2,\ \ \forall x\in\T^d,\xi,\tilde\zeta\in\R^d.$$
Therefore, we have
\begin{align}\label{ineq:first}
\int_0^T\int_{\T^d}c_H(1+\e\eta'(t))&\left|j_1(\nabla u^\e)-j_2\left(-\frac{w}{(1+\e\eta'(t))m}\right)\right|^2m\dd x\dd t\\
\nonumber&\le\int_0^T\int_{\T^d}\left[H^\e(t,x,\nabla u^\e)+(H^\e)^*(t,x,-w/m)+\nabla u^\e\cdot w/m \right]m\dd x\dd t
\end{align}
and similarly
\begin{align}\label{ineq:second}
\int_0^T\int_{\T^d}c_H(1-\e\eta'(t))&\left|j_1(\nabla u^{-\e})-j_2\left(-\frac{w}{(1-\e\eta'(t))m}\right)\right|^2m\dd x\dd t\\
\nonumber&\le\int_0^T\int_{\T^d}\left[H^{-\e}(t,x,\nabla u^{-\e})+(H^{-\e})^*(t,x,-w/m)+\nabla u^\e\cdot w/m \right]m\dd x\dd t.
\end{align}
By the triangle inequality,
\begin{align}\label{ineq:chain}
&\int_0^T\int_{\T^d}\frac{c_H}{3}\min\{1+\e\eta'(t);1-\e\eta'(t)\}\left|j_1(\nabla u^\e)-j_1(\nabla u^{-\e})\right|^2m\dd x\dd t\\
\nonumber&\le(1+\e)c_H \int_0^T\int_{\T^d}\left|j_1(\nabla u^\e)-j_2\left(-\frac{w}{(1+\e\eta'(t))m}\right)\right|^2m\dd x\dd t\\
\nonumber&+(1+\e)c_H \int_0^T\int_{\T^d}\left|j_2\left(-\frac{w}{(1+\e\eta'(t))m}\right)-j_2\left(-\frac{w}{(1-\e\eta'(t))m}\right)\right|^2m\dd x\dd t\\
\nonumber&+ (1+\e)c_H \int_0^T\int_{\T^d}\left|j_1(\nabla u^{-\e})-j_2\left(-\frac{w}{(1-\e\eta'(t))m}\right)\right|^2m\dd x\dd t,
\end{align}
where it remains to estimate the second term on the right-hand side.
For this we note that
\begin{align*}
&\left|j_2(\zeta/(1+\e\eta'(t)))-j_2(\zeta/(1-\e\eta'(t)))\right|^2=\left|j_2(\zeta(1-\e\eta'(t)+O(\e^2)))-j_2(\zeta(1+\e\eta'(t)+O(\e^2)))\right|^2\\
&=|D_\zeta j_2(\zeta(1-\e\eta'(t)+O(\e^2)))\cdot\zeta\eta'(t)|^2 \e^2\le C |\zeta|^{r'}\e^2,
\end{align*}
where the last constant depends only on $\eta'(t)$ and the constant in the hypothesis \eqref{hyp:j_2}. Setting $\zeta:=-w/m$ in the previous inequality, we find that the second term on the right-hand side of \eqref{ineq:chain} is $O(\e^2)$ since $\ds\int_0^T\int_{\T^d}\left|w/m\right|^{r'}m\dd x\dd t$ is finite.
Equation \eqref{eq:comb5} now becomes
\begin{multline} \label{eq:comb6}
\frac{c_H}{6}\int_0^T\int_{\T^d}\left|j_1(\nabla u^\e)-j_1(\nabla u^{-\e})\right|^2m\dd x\dd t
\\
\leq 
-\int_0^T\int_{\T^d}\left(f^\e(t,x,m^\e) - f(x,m)\right)(m^\e - m)\dd x\dd t
+ O(\e^2)
\end{multline}
for $\e$ small enough. \\

{\it Step 4. Time regularity for $m$.} 

We have
\begin{align*}
&-\int_0^T\int_{\T^d}\left(f^\e(t,x,m^\e)-f(x,m)\right)(m^\e-m)\dd x\dd t \\
&= - \iint_{\{m^\e \leq m\}} (f^\e(t,x,m^\e) - f(x,m^\e))(m^{\e} - m) \dd x \dd t 
- \iint_{\{m^\e \leq m\}} (f(x,m^\e) - f(x,m))(m^{\e} - m) \dd x \dd t\\
&- \iint_{\{m < m^\e\}} (f^\e(t,x,m^\e) - f^\e(t,x,m))(m^{\e} - m) \dd x \dd t
- \iint_{\{m < m^\e\}} (f^\e(t,x,m) - f(x,m))(m^{\e} - m) \dd x \dd t \\
&\leq C\int_0^T \int_{\bb{T}^d} |\e|\min\{(m^\e)^{q-1},m^{q-1}\}|m^{\e} - m| \dd x \dd t 
- c_0\int_0^T \int_{\bb{T}^d} \min\{(m^\e)^{q-2},m^{q-2}\}|m^\e - m|^2 \dd x \dd t \\
&\leq C|\e|^2 \int_0^T \int_{\bb{T}^d} \min\{m^\e,m\}^{q} \dd x \dd t 
- \frac{c_0}{2}\int_0^T \int_{\bb{T}^d} \min\{(m^\e)^{q-2},m^{q-2}\}|m^\e - m|^2 \dd x \dd t,
\end{align*}
where, we used Young's inequality in the last inequality, and the expression $\min\{(m^\e)^{q-2},m^{q-2}\}|m^\e - m|^2$ is  treated as zero whenever $m^\e = m$ (even in the case $q < 2$).
	Since
	$$
	\int_0^T \int_{\bb{T}^d} \min\{m^\e,m\}^{q} \dd x \dd t \leq \int_0^T \int_{\bb{T}^d} m^q \dd x \dd t  \leq C,
$$
Equation \eqref{eq:comb6} now becomes
$$
\frac{c_H}{6}\int_0^T\int_{\T^d}\left|j_1(\nabla u^\e)-j_1(\nabla u^{-\e})\right|^2m\dd x\dd t
+ \frac{c_0}{2}\int_0^T \int_{\bb{T}^d} \min\{(m^\e)^{q-2},m^{q-2}\}|m^\e - m|^2 \dd x \dd t = O(\e^2).
$$
Dividing the previous identity by $\e^2$ and letting $\e \to 0$, we conclude that $m^{1/2}\partial_tj_1(\nabla u), \ \partial_t(m^{q/2})\in L^2_{{\rm{loc}}}((0,T);L^2(\T^d))$, with norms estimated by a constant depending only on the data.
The proof is complete.
\end{proof}

\section{Open questions and further directions} \label{sec:open questions}

$\bullet$ An interesting direction of study would be the relaxation of the joint assumption on $q$ and $r$ (the growth exponent of $F$ and $H$, respectively) in the case of the planning problem. This should somehow imply also a more precise link between our work and the results of Orrieri-Porretta-Savar\'e in \cite{OrrPorSav}. This direction would be strongly related to the search for higher order summability estimates on the $m$ variable, also in the spirit of \cite{LavSan}. 

\vspace{10pt}

$\bullet$ The well-posedness of the second order (degenerate) planning problem is largely open (except the non-degenerate case with essentially quadratic Hamiltonians in \cite{Lions_course_planning, Por13,Por14}). Our hope is that the variational approach developed in the present paper and exploited also in \cite{OrrPorSav} will provide hints to attack the general second order problem.

 In this direction the exact controllability problem of the Fokker-Planck equation with general initial and final conditions (pointed out also by Lions in his lectures) seems to be an interesting open question, by its own. 

\vspace{10pt}

$\bullet$ We are aiming to pursue the global in time a priori Sobolev estimates in the time variable for the solutions  of  both mean field games and the planning problem, which so far seem to be inaccessible by relying only on our current techniques.

\vspace{1cm}

\bibliographystyle{alpha}
\bibliography{jameson_alpar}{}

\end{document}